\documentclass[reqno]{amsart}
\usepackage{amssymb}
\usepackage{graphicx}

\usepackage[usenames, dvipsnames]{color}
\usepackage{verbatim}
\usepackage{mathrsfs}
\usepackage{bm}
\usepackage{cite}

\numberwithin{equation}{section}

\newtheorem{theorem}{Theorem}[section]

\newtheorem{lemma}[theorem]{Lemma}
\newtheorem{prop}[theorem]{Proposition}
\newtheorem*{thmA}{Theorem A}

\theoremstyle{definition}
\newtheorem{remark}[theorem]{Remark}

\theoremstyle{definition}

\theoremstyle{definition}

\makeatletter
\def\dashint{\operatorname%
{\,\,\text{\bf-}\kern-.98em\DOTSI\intop\ilimits@\!\!}}
\makeatother

\newcommand{\RN}[1]{%
  \textup{\uppercase\expandafter{\romannumeral#1}}%
}

\renewcommand{\epsilon}{\varepsilon}

\newcounter{marnote}

\begin{document}


\title[Lower Bounds of the Gradient Estimates]{Lower bounds of gradient's blow-up for the Lam\'{e} system with partially infinite coefficients}

\author[H.G. Li]{Haigang Li}
\address[H. Li]{School of Mathematical Sciences, Beijing Normal University, Laboratory of MathematiCs and Complex Systems, Ministry of Education, Beijing 100875, China.}
\email{hgli@bnu.edu.cn.}
\thanks{The author was partially supported by  NSFC (11571042, 11631002), Fok Ying Tung Education Foundation (151003).}




\maketitle

\begin{abstract}
In composite material, the stress may be arbitrarily large in the narrow region between two close-to-touching hard inclusions. The stress is represented by the gradient of a solution to the Lam\'{e} system of linear elasticity. The aim of this paper is to establish lower bounds of the gradients of solutions of the Lam\'{e} system with partially infinite coefficients as the distance between the surfaces of discontinuity of the coefficients of the system tends to zero. Combining it with the pointwise upper bounds obtained in our previous work, the optimality of  the blow-up rate of gradients is proved for inclusions with arbitrary shape in dimensions two and three. The key to show this is that we find a blow-up factor, a linear functional of the boundary data, to determine whether the blow-up will occur or not.
\end{abstract}

\section{Introduction and main results}

In high-contrast composites, it is a common phenomenon that high concentration of extreme mechanical loads occurs in the narrow regions between two adjacent inclusions. Extreme loads are always amplified by such microstructure, which will cause failure or fracture initiation. Stimulated by the well-known work of Babu\u{s}ka, et al. \cite{ba}, where the Lam\'{e} system of linear elasticity was assumed and the initiation of damage and fracture in composite materials was computationally analyzed, we consider the Lam\'{e} system with partially infinite coefficients to characterize high-contrast composites. The gradient of the solution exhibits singular behavior with respect to the distance between hard inclusions. This paper is a continuation of \cite{bll,bll2}, where a pointwise upper bound of the gradient of solution is established by an iteration technique with respect to the energy, as the distance (say, $\varepsilon$) between the surfaces of discontinuity of the coefficients of the system tends to zero. 

The main purpose of this paper is to show the blow-up rates obtained in \cite{bll,bll2} are actually optimal, by establishing the lower bounds on the gradients of solutions of the Lam\'{e} system with partially infinite coefficients in two physically relevant dimensions $d=2,3$. Namely, the optimal blow-up rates are, respectively, $\varepsilon^{-1/2}$ in dimension $d=2$, $(\varepsilon|\log\varepsilon|)^{-1}$  in dimension $d=3$. Usually, it is not easy to obtain a lower bound. The novelty of this paper is that we introduce a blow-up factor defined by a solution of the limit case when two inclusions touch each other, which is a linear functional of the boundary data to determine whether or not the blow-up to occur. Physically, this factor seems much natural. Here new difficulties need to be overcome, and a number of refined estimates are used in our proof. The introduced methodology allows us define an analogous blow-up factor for the perfect conductivity problem considered in \cite{bly1} and give a new and simple proof for the lower bound estimates.

There have been many works on the analogous question for the following scalar equation with Dirichlet boundary condition, called the conductivity problem, 
\begin{equation}\label{equk}
\begin{cases}
\nabla\cdot\Big(a_{k}(x)\nabla{u}_{k}\Big)=0&\mbox{in}~\Omega,\\
u_{k}=\varphi&\mbox{on}~\partial\Omega,
\end{cases}
\end{equation}
where $\Omega$ is a bounded open set of $\mathbb{R}^{d}$, $d\geq2$, containing two $\varepsilon$-apart convex inclusions $D_{1}$ and $D_{2}$, $\varphi\in{C}^{2}(\partial\Omega)$ is given, and
$$a_{k}(x)=
\begin{cases}
k\in(0,\infty)&\mbox{in}~D_{1}\cup{D}_{2},\\
1&\mbox{in}~\Omega\setminus\overline{D_{1}\cup{D}_{2}}.
\end{cases}
$$
For touching disks $D_{1}$ and $D_{2}$ in dimension $d=2$, Bonnetier and Vogelius \cite{bv} first proved that $|\nabla u_{k}|$ remains bounded. The bound depends on the value of $k$. Li and Vogelius \cite{lv} extended the result to general divergence form second order elliptic equations with piecewise smooth coefficients in all dimensions, and they proved that $|\nabla u|$ remains bounded as $\varepsilon\rightarrow0$. Li and Nirenberg \cite{ln} further extended  the results in \cite{lv} to general divergence form second order elliptic systems including systems of elasticity.

The estimates in \cite{ln} and \cite{lv} depend on the ellipticity of the coefficients. If ellipticity constants are allowed to deteriorate, the situation is quite different. When $k=\infty$, the $L^\infty$-norm of $|\nabla u_\infty|$ for the solutions $u_{\infty}$ of  (\ref{equk}) generally becomes unbounded as $\varepsilon$ tends to $0$.  The blow-up rate of $|\nabla u_\infty|$ is respectively
 $\varepsilon^{-1/2}$ in dimension $d=2$,
$(\varepsilon|\log\varepsilon|)^{-1}$ in
dimension $d=3$,  and  $\varepsilon^{-1}$  in dimension $d\ge 4$. See
Bao, Li and Yin \cite{bly1}, as well as Budiansky and Carrier \cite{bc},
  Markenscoff \cite{m}, Ammari, Kang and Lim \cite{akl},
 Ammari, Kang, Lee, Lee and Lim \cite{aklll} and Yun  \cite{y1,y2}.
Further, more detailed, characterizations of the singular behavior of $\nabla{u}_{\infty}$ have been obtained by Ammari, Ciraolo, Kang, Lee and Yun \cite{ackly}, Ammari, Kang, Lee, Lim and Zribi \cite{AKLLZ}, Bonnetier and Triki \cite{bt0, bt} and
 Kang, Lim and Yun \cite{kly, kly2}. For related works, see \cite{abtv, adkl, agkl, akkl, bjl,bly2, bt, bcn, dong, dl, dongzhang, g, gb, gn, jlx, kleey, keller1, keller2, Li-Li, llby, lx, ly, ly2,mnp,y3} and the references therein.

We follow the notations of \cite{bll,bll2}. Let $\Omega\subset\mathbb{R}^{d}$ be a bounded open set with $C^{2}$ boundary, and $D_{1}$ and $D_{2}$ are two disjoint convex open sets in $\Omega$ with
$C^{2,\gamma}$ boundaries, $0<\gamma<1$, which are $\varepsilon$ apart and far away from $\partial{\Omega}$, that is,
\begin{equation}\label{omega}
\begin{array}{l}
\displaystyle \overline{D}_{1},\overline{D}_{2}\subset\Omega,\quad\mbox{the principle curvatures of }\partial{D}_{1},\partial{D}_{2}\geq\kappa_{0}>0,\\
\displaystyle \varepsilon:=\mathrm{dist}(D_{1},D_{2})>0,\quad \mathrm{dist}(D_{1}\cup{D}_{2},\partial{\Omega})>\kappa_{1}>0,
\end{array}
\end{equation}
where $\kappa_{0},\kappa_{1}$ are constants independent of $\varepsilon$. We also assume that the $C^{2,\gamma}$ norms of $\partial{D}_{i}$ are bounded by some constant independent of $\varepsilon$. This implies that each $D_{i}$ contains a ball of radius $r_{0}^{*}$ for some constant $r_{0}^{*}>0$ independent of $\varepsilon$. Denote $$\widetilde{\Omega}:=\Omega\setminus\overline{D_{1}\cup{D}_{2}}.$$
Assume that $\widetilde{\Omega}$ and $D_{1}\cup{D}_{2}$ are occupied, respectively, by two different isotropic and homogeneous materials with different Lam\'{e} constants $(\lambda,\mu)$ and $(\lambda_{1},\mu_{1})$. Then the elasticity tensors for the inclusions and the background can be written, respectively, as  $\mathbb{C}^{1}$ and $\mathbb{C}^{0}$, with
$$C_{ij\,kl}^{1}=\lambda_{1}\delta_{ij}\delta_{kl}+\mu_{1}(\delta_{ik}\delta_{jl}+\delta_{il}\delta_{jk}),$$
and
\begin{equation*}\label{C0}
C_{ij\,kl}^{0}=\lambda\delta_{ij}\delta_{kl}+\mu(\delta_{ik}\delta_{jl}+\delta_{il}\delta_{jk}),
\end{equation*}
where $i,j,k,l=1,2,3$ and $\delta_{ij}$ is the Kronecker symbol: $\delta_{ij}=0$ for $i\neq{j}$, $\delta_{ij}=1$ for $i=j$.
Let $u=\left(u_{1},u_{2},\cdots, u_{d}\right)^{T}:~\Omega\rightarrow\mathbb{R}^{d}$ denote the displacement field.
For a given vector valued function $\varphi$, we consider the following Dirichlet problem for the Lam\'{e} system
\begin{equation}\label{system}
\begin{cases}
\nabla\cdot\bigg(\left(\chi_{\widetilde{\Omega}}\mathbb{C}^{0}+\chi_{D_{1}\cup{D}_{2}}\mathbb{C}^{1}\right)e(u)\bigg)=0,
&\mbox{in}~\Omega,\\
u=\mathbf{\varphi},&\mbox{on}~\partial{\Omega},
\end{cases}
\end{equation}
where $\chi_{D}$ is the characteristic function of $D\subset\mathbb{R}^{d}$,
$$e(u)=\frac{1}{2}\left(\nabla{u}+(\nabla{u})^{T}\right)$$
is the strain tensor.

Assume that the standard ellipticity condition holds for \eqref{system}, that is,
\begin{equation}\label{coeff3_convex}
\mu>0,\quad\,d\lambda+2\mu>0;\quad\quad\mu_{1}>0,\quad\,d\lambda_{1}+2\mu_{1}>0.
\end{equation}
For $\varphi\in{H}^{1}(\Omega;\mathbb{R}^{d})$, it is well known that there exists a unique solution $u\in{H}^{1}(\Omega;\mathbb{R}^{d})$ of the Dirichlet problem \eqref{system}. More details can be found in the Appendix in \cite{bll}.

Let
\begin{equation*}\label{def_Psid}
\Psi:=\bigg\{\psi\in{C}^{1}(\mathbb{R}^{d};\mathbb{R}^{d})~\big|~2e(\psi)=\nabla\psi+(\nabla\psi)^{T}=0~\bigg\}
\end{equation*}
be the linear space of rigid displacement in $\mathbb{R}^{d}$. With $e_{1},\cdots,e_{d}$ denoting the standard basis of $\mathbb{R}^{d}$, $$\left\{~e_{i},~x_{j}e_{k}-x_{k}e_{j}~\big|~1\leq\,i\leq\,d,~1\leq\,j<k\leq\,d~\right\}$$ is a basis of $\Psi$. Denote this basis of $\Psi$ as $\{\psi^{\alpha}\}$, $\alpha=1,2,\cdots,\frac{d(d+1)}{2}$. If $\xi\in{H}^{1}(D; \mathbb{R}^{d})$, $e(\xi)=0$ in $D$,
and $D\subset \mathbb{R}^{d}$ is a connected  open set, then $\xi$ is a linear combination of $\{\psi^{\alpha}\}$ in $D$. 

For fixed $\lambda$ and $\mu$ satisfying \eqref{coeff3_convex}, denoting $u_{\lambda_{1},\mu_{1}}$ the solution of \eqref{system}. Then, as proved in the Appendix in \cite{bll},
$$u_{\lambda_{1},\mu_{1}}\rightarrow\,u\quad\mbox{in}~H^{1}(\Omega;\mathbb{R}^{d})\quad
\mbox{as}~~\min\{\mu_{1},d\lambda_{1}+2\mu_{1}\}\rightarrow\infty,$$
where $u$ is a $H^{1}(\Omega;\mathbb{R}^{d})$ solution of
\begin{equation}\label{mainequation}
\begin{cases}
\mathcal{L}_{\lambda,\mu}u:=\nabla\cdot\left(\mathbb{C}^{0}e(u)\right)=0,&\mbox{in}~\widetilde{\Omega},\\
u\big|_{+}=u\big|_{-},&\mbox{on}~\partial{D}_{1}\cup\partial{D}_{2},\\
e(u)=0,&\mbox{in}~D_{1}\cup{D}_{2},\\
\int_{\partial{D}_{i}}\frac{\partial{u}}{\partial\nu}\big|_{+}\cdot\psi^{\beta}=0,&i=1,2,~~\beta=1,2,\cdots,\frac{d(d+1)}{2},\\
u=\varphi,&\mbox{on}~\partial{\Omega}.
\end{cases}
\end{equation}
where
$$\frac{\partial{u}}{\partial\nu}\Big|_{+}:=\left(\mathbb{C}^{0}e(u)\right)\vec{n}
=\lambda\left(\nabla\cdot{u}\right)\vec{n}+\mu\left(\nabla{u}+(\nabla{u})^{T}\right)\vec{n},$$
and $\vec{n}$ is the unit outer normal of $D_{i}$, $i=1,2$. Here and throughout this paper the subscript $\pm$ indicates the limit from outside and inside the domain, respectively. 

The existence, uniqueness and regularity of weak solutions of \eqref{mainequation}, as well as a variational formulation, can be found in the Appendix in \cite{bll}. In particular, the  $H^{1}$ weak solution is in ${C}^{1}(\overline{\widetilde{\Omega}};\mathbb{R}^{d})\cap{C}^{1}(\overline{D_{1}\cup{D}_{2}};\mathbb{R}^{d})$. The solution is the unique
function which has the least energy in appropriate functional spaces, that is,
$$
I_{\infty}[u]=\min_{v\in\mathcal{A}}I_{\infty}[v],
\qquad~\quad
I_{\infty}[v]:=
\frac{1}{2}\int_{\widetilde \Omega}\left(\mathbb{C}^{(0)}
e(v),e(v)\right)dx,
$$
where
$$
\mathcal{A}:=\left\{u\in{H}_{\varphi}^{1}(\Omega;\mathbb{R}^{d})
~\big|~e(u)=0~~\mbox{in}~D_1\cup D_2\right\}.
$$

It is well known, see \cite{osy},  that for any open set $O$ and $u,v\in{C}^{2}(O)$,
\begin{equation}\label{equ4-1}
\int_{O}\left(\mathbb{C}^{0}e(u),e(v)\right)dx
=-\int_{O}\left(\mathcal{L}_{\lambda,\mu}u\right)\cdot{v}
+\int_{\partial O}\frac{\partial{u}}{\partial\nu}\Big|_{+}\cdot{v}.
\end{equation}
A calculation gives
\begin{align*}\label{L_u}
\left(\mathcal{L}_{\lambda,\mu}u\right)_{k}=\mu\Delta{u}_{k}+(\lambda+\mu)\partial_{x_{k}}\left(\nabla\cdot{u}\right),\quad\,k=1,2,3.
\end{align*}
We assume that for some $\delta_{0}>0$,
\begin{equation}\label{coeff4_strongelyconvex}
\delta_{0}\leq\mu,d\lambda+2\mu\leq\frac{1}{\delta_{0}}.
\end{equation}
 
Throughout the paper, unless otherwise stated, $C$ denotes a constant, whose values may vary from line to line, depending only on $d,\kappa_{0},\kappa_{1},\gamma,\delta_{0}$,  and an upper bound of the $C^{2}$ norm of $\partial\Omega$ and the $C^{2,\gamma}$ norms of $\partial{D}_{1}$ and $\partial{D}_{2}$, but independent of $\varepsilon$. Also, we call a constant having such dependence a {\it universal constant}. Let
$$\rho_{d}(\varepsilon)=\begin{cases}
\sqrt{\varepsilon},&d=2,\\
\frac{1}{|\log\varepsilon|}, &d=3.
\end{cases}$$
In order to show the optimality of the blow-up rate, we first recall the following upper bound estimates established in \cite{bll,bll2}.

\begin{thmA}\label{mainthm1}
(Upper Bounds, \cite{bll,bll2}) For $d=2,3$, assume that $\Omega$, $D_{1},D_{2}$, $\varepsilon$ are defined in \eqref{omega}, $\lambda$ and $\mu$ satisfy \eqref{coeff4_strongelyconvex} for some $\delta_{0}>0$,  and $\varphi\in{C}^{2}(\partial\Omega;\mathbb{R}^{d})$. Let $u\in{H}^{1}(\Omega;\mathbb{R}^{d})\cap{C}^{1}(\overline{\widetilde{\Omega}};\mathbb{R}^{d})$ be the solution of \eqref{mainequation}. Then for $0<\varepsilon<1/2$, we have
\begin{equation*}
\|\nabla{u}\|_{L^{\infty}(\Omega;\mathbb{R}^{d})}
\leq\,\frac{C\rho_{d}(\varepsilon)}{\varepsilon}\|\varphi\|_{C^{2}(\partial{\Omega};\mathbb{R}^{d})},
\end{equation*}
where $C$ is a universal constant.
\end{thmA}

\begin{remark}\label{rem1}
Since $D_{1}$ and $D_{2}$ are two strictly convex subdomains of $\Omega$, there exist two points $P_{1}\in\partial{D}_{1}$ and $P_{2}\in\partial{D}_{2}$ such that
\begin{equation}\label{P1P2}
\mathrm{dist}(P_{1},P_{2})=\mathrm{dist}(\partial{D}_{1},\partial{D}_{2})=\varepsilon.
\end{equation}
Use $\overline{P_{1}P_{2}}$ to denote the line segment connecting $P_{1}$ and $P_{2}$. The proof of Theorem A actually gives us the following stronger estimates for $x\in\widetilde{\Omega}$:
\begin{equation}\label{mainestimates}
|\nabla{u}(x)|\leq
\begin{cases}
\dfrac{C}{\sqrt{\varepsilon}+\mathrm{dist}(x,\overline{P_{1}P_{2}})}\|\varphi\|_{C^{2}(\partial{\Omega};\mathbb{R}^{d})},&d=2;\\
\bigg(\dfrac{C}{|\log\varepsilon|\big(\varepsilon+\mathrm{dist}^{2}(x,\overline{P_{1}P_{2}})\big)}
+\dfrac{C\mathrm{dist}(x,\overline{P_{1}P_{2}})}{\varepsilon+\mathrm{dist}^{2}(x,\overline{P_{1}P_{2}})}\bigg)
~\|\varphi\|_{C^{2}(\partial{\Omega};\mathbb{R}^{d})},&d=3.
\end{cases}
\end{equation}
\end{remark}

The pointwise upper bound in \eqref{mainestimates} shows that the gradient $|\nabla u(x)|$ (at least the right hand side of \eqref{mainestimates}) would achieve its maximum on the segment $\overline{P_{1}P_{2}}$ if the blow-up occurred. However, whether the blow-up occurs or not depends totally on the boundary data $\varphi$ for given domain $\Omega$, $D_{1}$ and $D_{2}$ with suitable smoothness. Therefore, in order to show that the blow-up rate of the gradients obtained in Theorem A is  optimal, it is necessary to establish the lower bound of $|\nabla u(x)|$ on the segment $\overline{P_{1}P_{2}}$ with the same blow-up rate. 

To this aim, a key ingredient is to find a function $u_{b}^{*}$, one part of the limit function of $u$ as $\varepsilon$ tends to zero. Denote $D_{1}^{*}:=\{~x\in\mathbb{R}^{d}~|~x+P_{1}\in{D}_{1}~\}$ and $D_{2}^{*}:=\{~x\in\mathbb{R}^{d}~|~x+P_{2}\in{D}_{2}~\}$. Set $\widetilde{\Omega}^{*}:=D\setminus \overline{D_{1}^{*}\cup{D}_{2}^{*}}$. Let $u_{b}^{*}$ be the solution of the following boundary value problem:
\begin{align}\label{u00*}
\begin{cases}
  \mathcal{L}_{\lambda,\mu}u_{b}^{*}=0,\quad&
\hbox{in}\  \widetilde{\Omega}^{*},  \\
u_{b}^{*}=\sum\limits_{\alpha=1}^{d}C_{*}^{\alpha}\psi^{\alpha},\ &\hbox{on}\ \partial{D}_{1}^{*}\cup\partial{D}_{2}^{*},\\
u_{b}^{*}=\varphi(x),&\hbox{on} \ \partial{D},
\end{cases}
\end{align}
where the constants $C_{*}^{\alpha}$, $\alpha=1,2,\cdots,d$, are determined later. 
We remark that $u_{b}^{*}$ is smooth near the origin by theorem 1.1 in \cite{llby}. In order to capture the lower bound of $|\nabla u|$, we now introduce a vector-valued linear functional of $\varphi$,
\begin{align}\label{def_bj*}
b_{*1}^{\beta}[\varphi]:=\int_{\partial{D}_{1}^{*}}\frac{\partial u_{b}^{*}}{\partial \nu}\Big|_{+}\cdot\psi^{\beta},\quad\beta=1,2,\cdots,d.
\end{align}
Notice that these quantities are independent of $\varepsilon$. It will be turn out that they will determine whether or not the blow-up to occur. We call them {\em{blow-up factors}}. Their important role will be shown in next section. The main result of this paper is the following lower bounds of $|\nabla u|$ on $\overline{P_{1}P_{2}}$.

\begin{theorem}(Lower Bounds for $d=2,3$).\label{thm1.2}
For $d=2,3$, under the assumptions as in Theorem A, let $u\in H^1(\Omega; \mathbb{R}^{d})\cap C^1(\overline{\widetilde{\Omega}}; \mathbb{R}^{d})$ be a solution to \eqref{mainequation}. Then if there exists a $\varphi$ such that $b_{*1}^{k_0}[\varphi]\neq0$ for some integer $1\leq\,k_0\leq\,d$, then for sufficiently small $0<\varepsilon<1/2$, 
$$\big|\nabla u(x)\big|\geq\frac{\rho_{d}(\varepsilon)}{C\varepsilon}|b_{*1}^{k_0}[\varphi]|,\quad\mbox{for}~~\,x\in\overline{P_{1}P_{2}},$$
where $C$ is a universal constant.
\end{theorem}

\begin{remark}
Theorem \ref{thm1.2}, together with Theorem A, shows that the optimal blow-up rate of $|\nabla{u}|$ is $\frac{\rho_{d}(\varepsilon)}{\varepsilon}$, namely, $\varepsilon^{-1/2}$ in dimension $d=2$, $(\varepsilon|\log\varepsilon|)^{-1}$ in dimension $d=3$. These generic blow-up rates are actually the same as the scalar case \cite{bly1}, as well as expected in \cite{bll,bll2}. Of course, we also can define $b_{*2}^{\beta}[\varphi]$ on the boundary $\partial{D}_{2}^{*}$. We would like to point out that Kang and Yu \cite{ky} proved the blow-up rate $\varepsilon^{-1/2}$ is optimal under a stronger assumption that inclusions are of $C^{3,\alpha}$ in dimension two by using a singular function. The method is totally different with ours. Here we only assume that $\partial{D}_{1}$ and $\partial{D}_{2}$ are of $C^{2,\gamma}$ as before.
\end{remark}

For the convenience of application, we give the corresponding results for the following perfect conductivity problem
\begin{equation}\label{mainequation2}
\begin{cases}
\Delta u=0,&\mbox{in}~\widetilde{\Omega},\\
u\big|_{+}=u\big|_{-},&\mbox{on}~\partial{D}_{1}\cup\partial{D}_{2},\\
\nabla u=0,&\mbox{in}~D_{1}\cup{D}_{2},\\
\int_{\partial{D}_{i}}\frac{\partial{u}}{\partial\nu}\big|_{+}=0,&i=1,2,~~\alpha=1,2,\cdots,\frac{d(d+1)}{2},\\
u=\varphi,&\mbox{on}~\partial{\Omega}.
\end{cases}
\end{equation}
The proof is much simpler and shorter than that for the elasticity case. An analogous blow-up factor is defined by
\begin{equation}\label{b1*}
b_{1}^{*}[\varphi]:=\int_{\partial{D}_{1}^{*}}\frac{\partial u^{*}}{\partial\nu},
\end{equation}
where $u^{*}$ satisfies
$$
\begin{cases}
\Delta u^{*}=0,&\mbox{in}~\widetilde{\Omega}^{*},\\
u^{*}=C^{*},&\mbox{on}~\partial{D}^{*}_{1}\cup\partial{D}^{*}_{2},\\
u^{*}=\varphi,&\mbox{on}~\partial{\Omega}.
\end{cases}$$
and the constant $C^{*}$ is uniquely determined by minimizing the energy 
$$\int_{\widetilde{\Omega}^{*}}|\nabla u|^{2}dx,$$
in an admissiable function space $$
\mathcal{A}_{0}:=\left\{u\in{H}^{1}(\Omega)
~\big|~\nabla u=0~~\mbox{in}~{D}^{*}_{1}\cup{D}^{*}_{2},~\mbox{and}~u=\varphi~\mbox{on}~\partial{\Omega}\right\}.$$

\begin{theorem}(Lower Bounds for perfect conductivity problem).\label{thm1.3}
For $d=2,3$, under the assumptions for the domain as in Theorem A, let $u\in H^1(\Omega)\cap C^1(\overline{\widetilde{\Omega}})$ be a solution to \eqref{mainequation2}. Then if there exists a $\varphi$ such that $b_{1}^{*}[\varphi]\neq 0$, then for sufficiently small $0<\varepsilon<1/2$, 
$$\big|\nabla u(x)\big|\geq\frac{\rho_{d}(\varepsilon)}{C\varepsilon}|b_{1}^{*}[\varphi]|,\quad\mbox{for}~~\,x\in\overline{P_{1}P_{2}},$$
where $C$ is a universal constant.
\end{theorem}

\begin{remark}
We remark that the quantity $b_{1}^{*}[\varphi]$ is independent of $\varepsilon$. This is an essential difference with $Q_{\varepsilon}[\varphi]$ defined in \cite{bly1}. On the other hand, one can see that $u^{*}$ is smooth near the origin while $v_{i}^{*}$, $i=1,2$, in the definition of $Q_{\varepsilon}[\varphi]$ are singular at the origin. So the definition of $b_{1}^{*}[\varphi]$, \eqref{b1*}, is more natural from physical viewpoint and it is easier to check whether it equals to zero or not, in these two physically related dimensions, although it can not be used to deal with higher dimensions cases so far.  
\end{remark}

The rest of this paper is organized as follows. In Section \ref{sec2}, we list several known results from \cite{bll,bll2} about $|\nabla v_{i}^{\alpha}|$ and $C_{i}^{\alpha}$, after we decompose the solution $u=\sum_{\alpha=1}^{d(d+1)/2}(C_{1}^{\alpha}v_{1}^{\alpha}+C_{2}^{\alpha}v_{2}^{\alpha})+v_{0}$ in $\widetilde{\Omega}$, see \eqref{decom_u} below. By a careful observation of the structure of a system of linear equations for $C_{i}^{\alpha}$, we find a quantity $b_{1}^{\beta}[\varphi]$, which turns out to be convergent to the blow-up factor $b_{*1}^{\beta}[\varphi]$. This is the heart of this paper. The proof is technical and carried out in Section \ref{sec3}. Section \ref{sec4} is devoted to proving Theorem \ref{thm1.2} in dimension two and the proof of Theorem \ref{thm1.2} in dimension three is given in Section \ref{sec5}. Finally, we prove Theorem \ref{thm1.3} in Section \ref{sec6} to give a new and more simple proof of results in \cite{bly1} for the perfect conductivity in dimension two and three, especially for the lower bound estimates of $|\nabla u|$.

\bigskip

\section{Preliminaries and the blow-up factor}\label{sec2}

In this section we first introduce a decomposition of the solution of \eqref{mainequation}. In Subsection \ref{subsec2.4}, we choose a new system of linear equations for $C_{i}^{\alpha}$ from the whole system to solve $C_{1}^{\alpha}-C_{2}^{\alpha}$, $\alpha=1,2,3$. It is a different way from the selection made in \cite{bll,bll2} that just allows us to obtain upper bound estimates. While this selection makes it possible to introduce the blow-up factor $b_{*1}^{\beta}[\varphi]$ in Subsection \ref{subsec2.5} to get a lower bound of $|\nabla u|$. In the end, we list several preliminary results from our earlier papers \cite{bll,bll2} to make our paper self-contained and our exposition clear.

\subsection{Decomposition of $u$}\label{sec_pre}

By the third line of \eqref{mainequation}, $u$ is a linear combination of $\{\psi^{\alpha}\}$ in $D_{1}$ and $D_{2}$, respectively. By using continuity, we decompose the solution of \eqref{mainequation}, as in \cite{bll},  as follows:
\begin{equation}\label{decom_u}
u=\sum_{\alpha=1}^{d(d+1)/2}C_{1}^{\alpha}v_{1}^{\alpha}+\sum_{\alpha=1}^{d(d+1)/2}C_{2}^{\alpha}v_{2}^{\alpha}+v_{0},\quad\quad\mbox{in}~\widetilde{\Omega},
\end{equation}
where $v_{i}^{\alpha}\in{C}^{1}(\widetilde{\Omega};\mathbb{R}^{d})$, $i=1,2$, $\alpha=1,2,\cdots,\frac{d(d+1)}{2}$,  and $v_{0}\in{C}^{1}(\widetilde{\Omega};\mathbb{R}^{d})$ are respectively the solution of
\begin{equation}\label{v1alpha}
\begin{cases}
\mathcal{L}_{\lambda,\mu}v_{i}^{\alpha}=0,&\mbox{in}~\widetilde{\Omega},\\
v_{i}^{\alpha}=\psi^{\alpha},&\mbox{on}~\partial{D}_{i},\\
v_{i}^{\alpha}=0,&\mbox{on}~\partial{D}_{j}\cup\partial{\Omega},~j\neq\,i,
\end{cases}
\end{equation}
and
\begin{equation}\label{v3}
\begin{cases}
\mathcal{L}_{\lambda,\mu}v_{0}=0,&\mbox{in}~\widetilde{\Omega},\\
v_{0}=0,&\mbox{on}~\partial{D}_{1}\cup\partial{D}_{2},\\
v_{0}=\varphi,&\mbox{on}~\partial{\Omega}.
\end{cases}
\end{equation}
The constants $C_{i}^{\alpha}:=C_{i}^{\alpha}(\varepsilon)$, $i=1,2$, $\alpha=1,2,\cdots,\frac{d(d+1)}{2}$, are uniquely determined by $u$.

By the decomposition \eqref{decom_u}, we write
\begin{equation}\label{nablau_dec}
\nabla{u}=\sum_{\alpha=1}^{d}\left(C_{1}^{\alpha}-C_{2}^{\alpha}\right)\nabla{v}_{1}^{\alpha}
+\sum_{i=1}^{2}\sum_{\alpha=d+1}^{\frac{d(d+1)}{2}}C_{i}^{\alpha}\nabla{v}_{i}^{\alpha}
+\nabla{u}_{b},\quad\mbox{in}~\widetilde{\Omega},
\end{equation}
where
$$u_{b}:=\sum_{\alpha=1}^{d}C_{2}^{\alpha}v^{\alpha}+v_{0}, \quad\,v^{\alpha}=v_{1}^{\alpha}+v_{2}^{\alpha}.$$
It is obvious that $v^{\alpha}$, $\alpha=1,2,\cdots,d$, verifies
\begin{equation}\label{equ_valpha}
\begin{cases}
\mathcal{L}_{\lambda,\mu}v^{\alpha}=0,&\mbox{in}~\widetilde{\Omega},\\
v^{\alpha}=\psi^{\alpha},&\mbox{on}~\partial{D}_{1}\cup\partial{D}_{2},\\
v^{\alpha}=0,&\mbox{on}~\partial{\Omega}.
\end{cases}
\end{equation}
Notice that from theorem 1.1 in \cite{llby}, 
\begin{equation}\label{v0bdd}
\|\nabla v^{\alpha}\|_{L^{\infty}(\widetilde{\Omega}^{*})}\leq\,C,\quad\mbox{and}~\|\nabla v_{0}\|_{L^{\infty}(\widetilde{\Omega}^{*})}\leq\,C,
\end{equation}
since the displacement takes the same constant value on the the boundaries of both inclusions. So that $|\nabla u_{b}|$ is also bounded.

\subsection{A selected system of linear equations for $C_{i}^{\alpha}$}\label{subsec2.4}

By the linearity of $e(u)$ and decomposition \eqref{nablau_dec},
$$e(u)=\sum_{\alpha=1}^{d}(C_{1}^{\alpha}-C_{2}^{\alpha})e\left(v_{1}^{\alpha}\right)+\sum_{i=1}^{2}\sum_{\alpha=d+1}^{\frac{d(d+1)}{2}}C_{i}^{\alpha}e\left(v_{i}^{\alpha}\right)+e({u}_{b}),
\quad\mbox{in}~~\widetilde{\Omega}.$$
It follows from the forth line of \eqref{mainequation} that
\begin{align}\label{C1C2_2}
\sum_{\alpha=1}^{d}(C_{1}^{\alpha}&-C_{2}^{\alpha})\int_{\partial{D}_{j}}\frac{\partial{v}_{1}^{\alpha}}{\partial\nu}\Big|_{+}\cdot\psi^{\beta}+\sum_{i=1}^{2}\sum_{\alpha=d+1}^{\frac{d(d+1)}{2}}C_{i}^{\alpha}\int_{\partial{D}_{j}}\frac{\partial{v}_{i}^{\alpha}}{\partial\nu}\Big|_{+}\cdot\psi^{\beta}\nonumber\\
&+\int_{\partial{D}_{j}}\frac{\partial{u}_{b}}{\partial\nu}\Big|_{+}\cdot\psi^{\beta}=0,\quad\,j=1,2,~~\beta=1,2,\cdots,\frac{d(d+1)}{2}.
\end{align}

Denote, for $i,j=1,2,~~\alpha,\beta=1,2,\cdots,\frac{d(d+1)}{2}$,
\begin{align}\label{def_bj}
a_{ij}^{\alpha\beta}:=-\int_{\partial{D}_{j}}\frac{\partial{v}_{i}^{\alpha}}{\partial\nu}\Big|_{+}\cdot\psi^{\beta},\qquad
b_{j}^{\beta}:=b_{j}^{\beta}[\varphi]=\int_{\partial{D}_{j}}\frac{\partial{u}_{b}}{\partial\nu}\Big|_{+}\cdot\psi^{\beta}.
\end{align}
Then \eqref{C1C2_2} can be written as
\begin{equation}\label{C1C2_3}
\left\{
\begin{aligned}
\sum_{\alpha=1}^{d}(C_{1}^{\alpha}-C_{2}^{\alpha})a_{11}^{\alpha\beta}+\sum_{i=1}^{2}\sum_{\alpha=d+1}^{\frac{d(d+1)}{2}}C_{i}^{\alpha}a_{i1}^{\alpha\beta}-b_{1}^{\beta}&=0,\\
\sum_{\alpha=1}^{d}(C_{1}^{\alpha}-C_{2}^{\alpha})a_{12}^{\alpha\beta}+\sum_{i=1}^{2}\sum_{\alpha=d+1}^{\frac{d(d+1)}{2}}C_{i}^{\alpha}a_{i2}^{\alpha\beta}-b_{2}^{\beta}&=0,
\end{aligned}
\right.\quad\quad~~\beta=1,2,\cdots,\frac{d(d+1)}{2}.
\end{equation}
We select $\beta=1,2,\cdots,\frac{d(d+1)}{2}$ for $j=1$ and $\beta=d+1,\cdots,\frac{d(d+1)}{2}$ for $j=2$ to solve $C_{1}^{\alpha}-C_{2}^{\alpha}$, $\alpha=1,2,\cdots,d$. One can see that this selection is different with that in \cite{bll,bll2}. For simplicity, we denote it in block matrix
\begin{equation*}\label{equ_abc}
AX:=
\begin{pmatrix}
A_{11}&A_{12}\\
A_{21}&A_{22}
\end{pmatrix}
\begin{pmatrix}
X_{1}\\
X_{2}
\end{pmatrix}
=\begin{pmatrix}
B_{1}\\
B_{2}
\end{pmatrix},
\end{equation*}
where 
$$A_{11}:=\begin{pmatrix}
      a_{11}^{\alpha\beta} 
    \end{pmatrix}_{\alpha,\beta=1,2,\cdots,d},
\qquad
A_{12}:=\begin{pmatrix}
      a_{11}^{\alpha\beta}&a_{12}^{\alpha\beta}
    \end{pmatrix}_{\alpha=1,2,\cdots,d;~\beta=d+1,\cdots,\frac{d(d+1)}{2}},\quad
$$
$$ A_{21}:=\begin{pmatrix}
     ~ a_{11}^{\alpha\beta}~ \\\\
      a_{21}^{\alpha\beta}
    \end{pmatrix}_{\alpha=d+1,\cdots,\frac{d(d+1)}{2};~\beta=1,2,\cdots,d},\qquad
 A_{22}:=\begin{pmatrix}
     ~ a_{11}^{\alpha\beta} &a_{12}^{\alpha\beta}~\\\\
      a_{21}^{\alpha\beta}&a_{22}^{\alpha\beta}
    \end{pmatrix}_{\alpha,\beta=d+1,\cdots,\frac{d(d+1)}{2}},
$$

$$X_{1}=\Big(C_{1}^{1}-C_{2}^{1}, C_{1}^{2}-C_{2}^{2},\cdots, C_{1}^{d}-C_{2}^{d}\Big)^{T},$$
$$\quad
X_{2}=\Big(C_{1}^{d+1},\cdots,C_{1}^{\frac{d(d+1)}{2}},C_{2}^{d+1},\cdots,C_{2}^{\frac{d(d+1)}{2}}\Big)^{T},$$
and
$$B_{1}=\Big(b_{1}^{1},b_{1}^{2},\cdots, b_{1}^{d}\Big)^{T},\qquad
B_{2}=\Big(b_{1}^{d+1},\cdots, b_{1}^{\frac{d(d+1)}{2}},b_{2}^{d+1},\cdots, b_{2}^{\frac{d(d+1)}{2}}\Big)^{T}.$$
Since $A$ is positive definition, we can solve $X_{1}$ by Cramer's rule. Then the quantities of $b_{1}^{\beta}$, $\beta=1,2,\cdots,d$, will play a key role to determine whether $C_{1}^{\alpha}-C_{2}^{\alpha}$, $\alpha=1,2,\cdots,d$, equal to zero or not. In next subsection, we will show that $b_{1}^{\beta}$ is convergent to the blow-up factor $b_{*1}^{\beta}$, for $\beta=1,2,\cdots,d$.

\subsection{$b_{1}^{\beta}[\varphi]$ convergent to the blow-up factor $b_{*1}^{\beta}[\varphi]$} \label{subsec2.5} 
Recalling the definitions of $b_{1}^{\beta}[\varphi]$ and $b_{*1}^{\beta}[\varphi]$, \eqref{def_bj} and \eqref{def_bj*}, respectively, we have $b_{1}^{\beta}\rightarrow b_{*1}^{\beta}$, as $\varepsilon\to0$, for $\beta=1,\cdots, d$,

\begin{prop}\label{lemma5.1d=2}
For $d=2,3$, and $\beta=1,2,\cdots,d$, 
\begin{align*}
\Big|b_{1}^{\beta}[\varphi]-b_{*1}^{\beta}[\varphi]\Big|\leq\,C\max\{\varepsilon^{1/3},\rho_{d}(\varepsilon)\}\left(\|\varphi\|_{L^{1}(\partial\Omega)}+|\partial\Omega|\right).
\end{align*}
Consequently, 
\begin{align*}
b_{1}^{\beta}[\varphi]\rightarrow b_{*1}^{\beta}[\varphi],\quad\hbox{as}\ \varepsilon\rightarrow 0,\quad\beta=1,2,\cdots,d.
\end{align*}
\end{prop}
To prove this convergence, similar to \eqref{equ_valpha} and \eqref{v3}, we define their limit cases, respectively, for $\alpha=1,2,\cdots,d$,
\begin{equation}\label{def_v*alpha}
\begin{cases}
\mathcal{L}_{\lambda,\mu}v^{*\alpha}=0,&\mbox{in}~\widetilde{\Omega}^{*},\\
v^{*\alpha}=\psi^{\alpha},&\mbox{on}~\partial{D}_{1}^{*}\cup\partial{D}_{2}^{*},\\
v^{*\alpha}=0,&\mbox{on}~\partial{\Omega}.
\end{cases}
\end{equation}
and
\begin{equation}\label{def_v0*}
\begin{cases}
\mathcal{L}_{\lambda,\mu}v_{0}^{*}=0,&\mbox{in}~\widetilde{\Omega}^{*},\\
v_{0}^{*}=0,&\mbox{on}~\partial{D}_{1}^{*}\cup\partial{D}_{2}^{*},\\
v_{0}^{*}=\varphi,&\mbox{on}~\partial{\Omega}.
\end{cases}
\end{equation}
Then
$$u_{b}^{*}=\sum_{\alpha=1}^{d}C_{*}^{\alpha}v^{*\alpha}+v_{0}^{*}.$$
It follows from theorem 1.1 in \cite{llby} that 
\begin{equation}\label{v0*bdd}
\|\nabla v^{*\alpha}\|_{L^{\infty}(\widetilde{\Omega}^{*})}\leq\,C,\quad\mbox{and}~\|\nabla v_{0}^{*}\|_{L^{\infty}(\widetilde{\Omega}^{*})}\leq\,C.
\end{equation}
So that $|\nabla u_{b}^{*}|$ is also bounded. We shall prove that $u_{b}^{*}$ is actually the limit of $u_{b}$ later. The proof of Proposition \ref{lemma5.1d=2} will be given in the next section.

To complete this section, we fix our notations and list several known results of \cite{bll,bll2} for later use. We use $x=(x', x_{d})$ to denote a point in $\mathbb{R}^{d}$, where $x'=(x_{1},x_{2},\cdots,x_{d-1})$. By a translation and rotation if necessary, we may assume without loss of generality that the points $P_{1}$ and $P_{2}$ in \eqref{P1P2} satisfy
$$P_{1}=\left(0',\frac{\varepsilon}{2}\right)\in\partial{D}_{1},\quad\mbox{and}\quad\,P_{2}=\left(0',-\frac{\varepsilon}{2}\right)\in\partial{D}_{2}.$$
Fix a small universal constant $R$, such that the portion of $\partial{D}_{1}$ and  $\partial{D}_{2}$ near $P_{1}$ and $P_{2}$, respectively, can be represented by
\begin{equation}\label{x3}
x_{d}=\frac{\varepsilon}{2}+h_{1}(x'),\quad\mbox{and}\quad\,x_{d}=-\frac{\varepsilon}{2}+h_{2}(x'), \quad\mbox{for}~~|x'|<2R.
\end{equation}
Then by the smoothness assumptions on $\partial{D}_{1}$ and $\partial{D}_{2}$, the functions $h_{1}(x')$ and $h_{2}(x')$ are of class $C^{2,\gamma}(B_{R}(0'))$,
satisfying
$$\frac{\varepsilon}{2}+h_{1}(x')>-\frac{\varepsilon}{2}+h_{2}(x'),\quad\mbox{for}~~|x'|<2R,$$
\begin{equation}\label{h1h20}
h_{1}(0')=h_{2}(0')=0,\quad\nabla{h}_{1}(0')=\nabla{h}_{2}(0')=0,
\end{equation}
\begin{equation}\label{h1h22}
\nabla^{2}h_{1}(0')\geq\kappa_{0}I,
\quad\,\nabla^{2}h_{2}(0')\leq-\kappa_{0}I,
\end{equation}
and
\begin{equation}\label{h1h2}
\|h_{1}\|_{C^{2,\gamma}(B_{2R}')}+\|h_{2}\|_{C^{2,\gamma}(B_{2R}')}\leq{C}.
\end{equation}
In particular, we only use a weaker relative strict convexity assumption of $\partial{D}_{1}$ and $\partial{D}_{2}$, that is
\begin{equation}\label{h1h23}
h_{1}(x')-h_{2}(x')\geq\kappa_{0}|x'|^{2},\quad\mbox{if}~~|x'|<2R.
\end{equation}

For $0\leq\,r\leq2R$, denote
$$\Omega_{r}:=\left\{~(x',x_{d})\in\mathbb{R}^{d}~\big|~-\frac{\varepsilon}{2}+h_{2}(x')<x_{d}<\frac{\varepsilon}{2}+h_{1}(x'),
~|x'|<r~\right\}.
$$For $|z'|\leq\,2R$, we always use $\delta$ to denote
\begin{equation*}
\delta:=\delta(z')=\varepsilon+h_{1}(z')-h_{2}(z').
\end{equation*}
By \eqref{h1h20}-\eqref{h1h23},
\begin{equation}\label{delta}
\frac{1}{C}\left(\varepsilon+|z'|^{2}\right)\leq\delta(z')\leq\,C\left(\varepsilon+|z'|^{2}\right).
\end{equation}

We now list the following estimates of $|\nabla v_{i}^{\alpha}|$ and $C_{i}^{\alpha}$ from \cite{bll,bll2}.
\begin{lemma}\label{prop_gradient}(\cite{bll,bll2})
Under the hypotheses of Theorem A, and the normalization $\|\varphi\|_{C^{2}(\partial\Omega)}=1$, let $v_{i}^{\alpha}$ and $v_{0}$ be the solution to \eqref{v1alpha} and \eqref{v3}, respectively. Then for $0<\varepsilon<1/2$, we have
\begin{align}
&\big\|\nabla{v}_{0}\big\|_{L^{\infty}(\widetilde{\Omega})}\leq\,C;\label{mainev0}\\
&\big\|\nabla v^{\alpha}\big\|_{L^{\infty}(\widetilde{\Omega})}
\leq\,C,\quad\alpha=1,2,\cdots,d;\label{mainev1+v2}\\
&\frac{1}{C(\varepsilon+|x'|^{2})}\leq\big|\nabla{v}_{i}^{\alpha}(x)\big|
\leq\frac{C}{\varepsilon+|x'|^{2}},\quad\,i=1,2,~~
\alpha=1,2,\cdots,d,~~x\in\widetilde{\Omega};\label{mainev1}\\
&\big|\nabla{v}_{i}^{\alpha}(x)\big|
\leq\,\frac{C|x'|}{\varepsilon+|x'|^{2}}+C,\quad\,i=1,2,
~~\alpha=d+1,\cdots,\frac{d(d+1)}{2},~~x\in\widetilde{\Omega};\label{mainevi3}
\end{align}
and
\begin{equation}\label{maineC}
\big|C_{i}^{\alpha}\big|\leq\,C,\quad
i=1,2,~\alpha=1,2,\cdots,\frac{d(d+1)}{2};
\end{equation}
for $d=2,3$,
\begin{equation}\label{maineC1-C2}
\big|C_{1}^{\alpha}-C_{2}^{\alpha}\big|\leq\,C\rho_{d}(\varepsilon),\quad
\alpha=1,2,\cdots,d.
\end{equation}
\end{lemma}

\begin{remark}
Estimate \eqref{maineC1-C2} can also be proved in Proposition \ref{prop4.1} and Proposition \ref{prop5.1} below in a different way. It tells us that as $\varepsilon\rightarrow0$, in dimensions two and three the difference $|C_{1}^{\alpha}-C_{2}^{\alpha}\big|\rightarrow0$, $\alpha=1,2,\cdots,d$,  which allows us to prove that $b_{1}^{\beta}[\varphi]$ can be convergent to the blow-up factor $b_{*1}^{\beta}[\varphi]$. But in higher dimensions $d\geq4$, so far we do not know whether $|C_{1}^{\alpha}-C_{2}^{\alpha}|$ tends to $0$ or not as $\varepsilon\rightarrow 0$. For more details, see the proof of Lemma \ref{lem3.2} in Section \ref{sec3}, where we prove that $C_{1}^{\alpha}$ and $C_{2}^{\alpha}$ have the same limit $C_{*}^{\alpha}$, for $\alpha=1,2,\cdots,d$. 
\end{remark}

\bigskip

\section{Proof of Proposition \ref{lemma5.1d=2}}\label{sec3}

We introduce a scalar auxiliary function $\bar{u}\in{C}^{2}(\mathbb{R}^{d})$ as before, such that $\bar{u}=1$ on
$\partial{D}_{1}$, $\bar{u}=0$ on
$\partial{D}_{2}\cup\partial\Omega$,
\begin{align}\label{ubar}
\bar{u}(x)
=\frac{x_{d}-h_{2}(x')+\frac{\varepsilon}{2}}{\varepsilon+h_{1}(x')-h_{2}(x')},\quad\mbox{in}~~\Omega_{2R},
\end{align}
and
\begin{equation*}\label{nablau_bar_outside}
\|\bar{u}\|_{C^{2}(\mathbb{R}^{d}\setminus\Omega_{R})}\leq\,C.
\end{equation*}
We use $\bar{u}$ to define vector-value auxiliary functions
\begin{equation*}\label{def:ubar1112}
\bar{u}_{1}^{\alpha}=\bar{u}\psi^{\alpha},\quad\alpha=1,2,\cdots,d,
\quad\,\mbox{in}~~\widetilde{\Omega}.
\end{equation*}
Thus, $\bar{u}_{1}^{\alpha}=v_{1}^{\alpha}$ on $\partial\widetilde{\Omega}$. Similarly, we define
\begin{equation*}\label{def:ubar2122}
\bar{u}_{2}^{\alpha}=\underline{u}\psi^{\alpha},\quad\alpha=1,2,\cdots,d,
\quad\,\mbox{in}~~\widetilde{\Omega},
\end{equation*}
such that $\bar{u}_{2}^{\alpha}=v_{2}^{\alpha}$ on $\partial\widetilde{\Omega}$, where $\underline{u}$ is a scalar function in ${C}^{2}(\mathbb{R}^{d})$ satisfying $\underline{u}=1$ on
$\partial{D}_{2}$, $\underline{u}=0$ on
$\partial{D}_{1}\cup\partial\Omega$, $\underline{u}(x)
=1-\bar{u}$, in $\Omega_{2R}$, and
$\|\underline{u}\|_{C^{2}(\mathbb{R}^{d}\setminus\Omega_{R})}\leq\,C$.

A direct calculation gives, in view of \eqref{h1h20}-\eqref{h1h23}, that
\begin{equation}\label{nablau_bar}
|\partial_{x_{k}}\bar{u}(x)|\leq\frac{C|x_{k}|}{\varepsilon+|x'|^{2}},~~k=1,2,\cdots,d-1,\quad
\partial_{x_{d}}\bar{u}(x)=\frac{1}{\delta(x')},\quad~x\in\Omega_{R}.
\end{equation}
Thus, for $i=1,2$, $\alpha=1,2,\cdots,d$,
\begin{equation}\label{nabla_ubar}
|\nabla_{x'}\bar{u}_{i}^{\alpha}(x)|\leq\frac{C}{\sqrt{\varepsilon+|x'|^{2}}},\quad\mbox{and}~\frac{1}{C(\varepsilon+|x'|^{2})}\leq\partial_{x_{d}}\bar{u}_{i}^{\alpha}(x)\leq\frac{C}{\varepsilon+|x'|^{2}},\quad~x\in\Omega_{R}.
\end{equation}

We need the following Lemma to prove Proposition \ref{lemma5.1d=2}.
\begin{lemma}\label{prop1}(\cite{bll,bll2})
Assume the above, let $v_{i}^{\alpha}\in{H}^1(\widetilde{\Omega};\mathbb{R}^{d})$ be the
weak solution of \eqref{v1alpha} with $\alpha=1,2,\cdots,d$. Then for $i=1,2,~\alpha=1,2,\cdots,d$, 
\begin{equation}\label{nabla_w_ialpha}
\left|\nabla(v_{i}^{\alpha}-\bar{u}_{i}^{\alpha})(x)\right|\leq
\begin{cases}
\displaystyle \frac{C}{\sqrt{\varepsilon}},&|x'|\leq\sqrt{\varepsilon},\\
\displaystyle \frac{C}{|x'|},&\sqrt{\varepsilon}<|x'|\leq\,R,
\end{cases}\quad\forall~~x\in\Omega_{R},
\end{equation}
and
\begin{equation}\label{nabla_w_out}
\|\nabla(v_{i}^{\alpha}-\bar{u}_{i}^{\alpha})\|_{L^{\infty}(\widetilde{\Omega}\setminus\Omega_{R})}\leq\,C.
\end{equation}
\end{lemma}

Let $u^{*}$ be the solution of the following Dirichlet boundary value problem:
\begin{align}\label{u00*}
\begin{cases}
  \mathcal{L}_{\lambda,\mu}u^{*}=0,\quad&
\hbox{in}\  \widetilde{\Omega}^{*},  \\
u^{*}=\sum\limits_{\alpha=1}^{d(d+1)/2}C_{*}^{\alpha}\psi^{\alpha},\ &\hbox{on}\ \partial{D}_{1}^{*}\cup\partial{D}_{2}^{*},\\
\int_{\partial{D}_{1}^{*}}\frac{\partial{u}^{*}}{\partial\nu}\big|_{+}\cdot\psi^{\beta}+\int_{\partial{D}_{2}^{*}}\frac{\partial{u}^{*}}{\partial\nu}\big|_{+}\cdot\psi^{\beta}=0,&\beta=1,2,\cdots,\frac{d(d+1)}{2},\\
u^{*}=\varphi(x),&\hbox{on} \ \partial{D},
\end{cases}
\end{align}
where $C_{*}^{\alpha}$, $\alpha=1,2,\cdots,d(d+1)/2$, are uniquely determined by minimizing the energy 
\begin{equation*}
\int_{\widetilde{\Omega}^{*}}\left(\mathbb{C}^{(0)}
e(v),e(v)\right)dx
\end{equation*}
in an admission function space $$
\mathcal{A}:=\left\{v\in{H}^{1}(\Omega; \mathbb{R}^{d})
~\big|~e(v)=0~~\mbox{in}~{D}^{*}_{1}\cup{D}^{*}_{2},~\mbox{and}~v=\varphi~\mbox{on}~\partial{\Omega}\right\},$$
and $u^{*}$ is the limit of $u$ in the sense of variation.
If $\varphi=0$, then by a variational argument, there is only trivial solution $u^{*}\equiv0$ for \eqref{u00*}. Hence, $C_{*}^{\alpha}\equiv0$, $\alpha=1,2,\cdots,d(d+1)/2$. Generally, if $\varphi\neq0$, we have

\begin{lemma}\label{lem3.2} For $d=2,3$, under the hypotheses of Theorem A, and the normalization $\|\varphi\|_{C^{2}(\partial\Omega)}=1$, let $C_{i}^{\alpha}$ and $C_{*}^{\alpha}$ be defined in \eqref{decom_u} and \eqref{u00*}, respectively. Then
$$\left|\frac{1}{2}(C_{1}^{\alpha}+C_{2}^{\alpha})-C_{*}^{\alpha}\right|\leq\,C\rho_{d}(\varepsilon),\quad\alpha=1,2,\cdots,d.$$
Consequently, in view of \eqref{maineC1-C2},
\begin{equation}\label{C*-Ci}
|C_{*}^{\alpha}-C_{i}^{\alpha}|=\frac{1}{2}|C_{1}^{\alpha}-C_{2}^{\alpha}|+C\rho_{d}(\varepsilon)\leq\,C\rho_{d}(\varepsilon),\quad\,i=1,2,~~\alpha=1,2,\cdots,d.
\end{equation}
\end{lemma}

The proof will be given later. We first use it to prove Proposition \ref{lemma5.1d=2}.

\begin{proof}[Proof of Proposition \ref{lemma5.1d=2}]

We here prove the case $\beta=1$ for instance. The  other cases are the same. Set
\begin{align}\label{b11decom}
b_{1}^{1}=\int_{\partial{D}_{1}}\frac{\partial u_{b}}{\partial \nu}\Big|_{+}\cdot \psi^{1}=&\int_{\partial{D}_{1}}\frac{\partial v_0}{\partial \nu}\Big|_{+}\cdot \psi^{1}+\sum_{\alpha=1}^{d}C_{2}^{\alpha}\int_{\partial{D}_{1}}\frac{\partial v^{\alpha}}{\partial \nu}\Big|_{+}\cdot \psi^{1}\nonumber\\
:=&b_{1}^{1,0}+\sum_{\alpha=1}^{d}C_{2}^{\alpha}b_{1}^{1,\alpha}.
\end{align}

{\bf STEP 1.} First, for $b_{1}^{1,0}$. It follows from the definitions of $v_0$ and $v_{1}^{1}$ and the integration by parts formula (\ref{equ4-1}) that
\begin{align*}
b_1^{1,0}&=\int_{\partial{D}_{1}}\frac{\partial v_0}{\partial \nu}\Big|_{+}\cdot v_{1}^{1}=\int_{\widetilde{\Omega}}\left(\mathbb{C}^0e(v_{1}^{1}), e(v_0)\right)=\int_{\partial\Omega}\frac{\partial v_{1}^{1}}{\partial \nu}\Big|_{+}\cdot \varphi.\\
\end{align*}
Similarly,
\begin{align*}
b_{*1}^{1,0}&:=\int_{\partial{D}_{1}^{*}}\frac{\partial v_{0}^{*}}{\partial \nu}\Big|_{+}\cdot\psi^{1}=\int_{\partial\Omega}\frac{\partial v_{1}^{*1}}{\partial \nu}\Big|_{+}\cdot \varphi,\\
\end{align*}
where $v_{1}^{*1}$ satisfies
\begin{align}\label{u1*}
\begin{cases}
  \mathcal{L}_{\lambda,\mu}v_{1}^{*1}=0,\quad&
\hbox{in}\  \widetilde\Omega^{*},  \\
v_{1}^{*1}=\psi^{1},\ &\hbox{on}\ \partial{D}_{1}^{*}\setminus\{0\},\\
v_{1}^{*1}=0,&\hbox{on} \ \partial{D}_{2}^{*}\cup\partial\Omega.
\end{cases}
\end{align}
Thus,
\begin{equation}\label{b10}
b_{1}^{1,0}-b_{1}^{*1,0}=\int_{\partial\Omega}\frac{\partial (v_{1}^{1}-v_{1}^{*1})}{\partial \nu}\Big|_{+}\cdot \varphi.
\end{equation}
It suffices to estimate $|\nabla(v_{1}^{1}-v_{1}^{*1})$ on the boundary $\partial\Omega$.

{\bf STEP 1.1.} In order to estimate the difference $v_{1}^{1}-v_{1}^{*1}$, we introduce two auxiliary functions 
\begin{gather*} \bar{u}_{1}^{1}=\begin{pmatrix}
~\bar{u}~\\
0\\
\vdots\\
0\end{pmatrix},\qquad\mbox{and}\quad~~ \bar{u}_{1}^{*1}=\begin{pmatrix}
\bar{u}^*\\
0\\
\vdots\\
0\end{pmatrix},
\end{gather*}
where $\bar{u}$ is defined by \eqref{ubar}, and $\bar{u}^*$ satisfies $\bar{u}^*=1$ on $\partial{D}_{1}^{*}\setminus\{0\}$, $\bar{u}^*=0$ on $\partial{D}_{2}^*\cup\partial\Omega$, and
\begin{align*}
\bar{u}^*=\frac{x_{d}-h_{2}(x')}{h_1(x')-h_{2}(x')},\quad\hbox{on}\ \Omega_{R}^{*},\qquad\|\bar{u}^*\|_{C^{2}(\widetilde\Omega^{*}\setminus\Omega_{R/2}^{*})}\leq\,C,
\end{align*}
where $\Omega^{*}_{r}:=\left\{~x\in\widetilde\Omega^{*}~\big|~h_{2}(x')<x_{d}<h_{1}(x'),
~ |x'|<r~\right\}$, for $r<R$. Similar to \eqref{nabla_ubar}, we have
\begin{equation}\label{nabla_ubar11*}
|\nabla_{x'}\bar{u}_{1}^{*1}(x)|\leq\frac{C}{|x'|},\quad\mbox{and}\quad~~~\frac{1}{C|x'|^{2}}\leq\partial_{x_{d}}\bar{u}_{1}^{*1}(x)\leq\frac{C}{|x'|^{2}},\quad~x\in\Omega_{R}^{*}.
\end{equation}
By making use of (\ref{h1h20}), (\ref{h1h22}), and \eqref{delta}, we obtain, for $x\in\Omega_{R}^{*}\setminus\{0\}$,
$$\left|~\nabla_{x'}\Big((\bar{u}_{1}^{1})_{1}-(\bar{u}_{1}^{*1})_{1}\Big)~\right|=\left|\nabla_{x'}(\bar{u}-\bar{u}^{*})\right|\leq\frac{C}{|x'|},
$$
and
\begin{align}\label{estimate1a_d=2}
\left|\partial_{x_{d}}\Big((\bar{u}_{1}^{1})_{1}-(\bar{u}_{1}^{*1})_{1}\Big)\right|&=\Big|\frac{1}{\varepsilon+h_1(x')-h_{2}(x')}-\frac{1}{h_1(x')-h_{2}(x')}\Big|\nonumber\\
&\leq\frac{C\varepsilon}{|x'|^2(\varepsilon+|x'|^{2})}.
\end{align}

Applying Lemma \ref{prop1} to (\ref{u1*}) leads to
\begin{align}\label{u1*-tildeu1*_d=2}
|\nabla (v_{1}^{*1}-\bar{u}_{1}^{*1})(x)|\leq \frac{C}{|x'|},\quad   x\in\Omega_{R}^{*};
\end{align}
and in view of \eqref{nabla_ubar11*},
\begin{align}\label{u-1*}
|\nabla_{x'}v_{1}^{*1}(x)|\leq \frac{C}{|x'|},\quad |\partial_{x_{d}}v_{1}^{*1}(x)|\leq \frac{C}{|x'|^2},\quad x\in\Omega_{R}^{*}.
\end{align}

{\bf STEP 1.2.} Next, we estimate the difference $v_{1}^{1}-v_{1}^{*1}$. Notice that $v_{1}^{1}-v_{1}^{*1}$ satisfies
\begin{align*}
\begin{cases}
  \mathcal{L}_{\lambda,\mu}(v_{1}^{1}-v_{1}^{*1})=0,\quad&
\hbox{in}\   V:=\Omega\setminus(\overline{D_{1}\cup D_{1}^{*}\cup D_{2}\cup D_{2}^{*}}),  \\
v_{1}^{1}-v_{1}^{*1}=\psi^{1}-v_{1}^{*1},\ &\hbox{on}\ \partial{D}_{i} \setminus D_{i}^{*},~i=1,2,\\
v_{1}^{1}-v_{1}^{*1}=v_{1}^{1}-\psi^{1},&\hbox{on} \ \partial{D}_{i}^{*}\setminus (D_{i}\cup\{0\}),~i=1,2,\\
v_{1}^{1}-v_{1}^{*1}=0,&\hbox{on} \ \partial\Omega.
\end{cases}
\end{align*}

Define a cylinder
$$\mathcal{C}_{r}:=\left\{x\in\mathbb{R}^{d}~\big|~|x'|<r,-\frac{\varepsilon}{2}+2\min_{|x'|=r}h_{2}(x')\leq\,x_{d}\leq\frac{\varepsilon}{2}+2\max_{|x'|=r}h_{1}(x')\right\},$$
for $r<R$. 
We first estimate $|v_{1}^{1}-v_{1}^{*1}|$ on $\partial(D_{1}\cup{D}_{1}^{*})\setminus\mathcal{C}_{\varepsilon^{\gamma}}$, where $0<\gamma<1/2$ to be determined later. For $\varepsilon$ sufficiently small, in view of the definition of $v_{1}^{*1}$,
\begin{align*}
|\partial_{x_{d}}v_{1}^{*1}(x)|\leq C,\quad  x\in\widetilde\Omega^{*}\setminus\Omega_{R}^{*}.
\end{align*}
By using mean value theorem, we have, for $x\in \partial{D}_{1} \setminus D_{1}^{*}$, 
\begin{align}\label{boundary1_d=2}
|(v_{1}^{1}-v_{1}^{*1})(x', x_{d})|&=|(\psi^{1}-v_{1}^{*1})(x', x_{d})|\nonumber\\
&=|v_{1}^{*1}(x',x_{d}-\varepsilon)-v_{1}^{*1}(x', x_{d})|\leq C\varepsilon.
\end{align}
For $x\in\partial{D}_{1}^{*}\setminus(D_{1}\cup \mathcal{C}_{\varepsilon^{\gamma}})$, using mean value theorem again and  (\ref{mainev1}),
\begin{align}\label{boundary2_d=2}
|(v_{1}^{1}-v_{1}^{*1})(x', x_{d})|&=|v_{1}^{1}(x',x_{d})-v_{1}^{1}(x', x_{d}+\varepsilon)|\nonumber\\
&\leq \frac{C\varepsilon}{\varepsilon+|x'|^2}\leq C\varepsilon^{1-2\gamma}.
\end{align}
Similarly, for $x\in \partial{D}_{2} \setminus D_{2}^{*}$, 
\begin{align}\label{boundaryD21}
|(v_{1}^{1}-v_{1}^{*1})(x', x_{d})|\leq C\varepsilon;
\end{align}
for $x\in\partial{D}_{2}^{*}\setminus(D_{2}\cup \mathcal{C}_{\varepsilon^{\gamma}})$, by (\ref{mainev1}),
\begin{align}\label{boundaryD22}
|(v_{1}^{1}-v_{1}^{*1})(x', x_{d})|\leq C\varepsilon^{1-2\gamma}.
\end{align}

For $x\in\Omega_{R}^{*}$ with $|x'|=\varepsilon^{\gamma}$, it follows from (\ref{nabla_w_ialpha}), (\ref{estimate1a_d=2}), and (\ref{u1*-tildeu1*_d=2}) that
\begin{align*}
\left|\partial_{x_{d}}(v_{1}^{1}-v_{1}^{*1})(x',x_{d})\right|
=&\,\left|\partial_{x_{d}}(v_{1}^{1}-\bar{u}_{1}^{1})
+\partial_{x_{d}}(\bar{u}_{1}^{1}-\bar{u}_{1}^{*1})+\partial_{x_{d}}(\bar{u}_{1}^{*1}-v_{1}^{*1})\right|(x',x_{d})\nonumber\\
\leq&\,\frac{C\varepsilon}{|x'|^{2}(\varepsilon+|x'|^{2})}+\frac{C}{|x'|}\nonumber\\
\leq&\,\frac{C}{\varepsilon^{4\gamma-1}}+\frac{C}{\varepsilon^{\gamma}}.
\end{align*}
Thus, for $x\in\Omega_{R}^{*}$ with $|x'|=\varepsilon^{\gamma}$, recalling \eqref{boundaryD22}, we have
\begin{align}\label{boundary3_d=2}
|(v_{1}^{1}-v_{1}^{*1})(x',x_{d})|=&|(v_{1}^{1}-v_{1}^{*1})(x',x_{d})-(v_{1}^{1}-v_{1}^{*1})(x',h_{2}(x'))|+C\varepsilon^{1-2\gamma}\nonumber\\
\leq&\left|\partial_{x_{d}}(v_{1}^{1}-v_{1}^{*1})\right|_{|x'|=\varepsilon^{\gamma}}\cdot (h_1(x')-h_{2}(x'))+C\varepsilon^{1-2\gamma}\nonumber\\
\leq&(\frac{C}{\varepsilon^{4\gamma-1}}+\frac{C}{\varepsilon^{\gamma}})\cdot \varepsilon^{2\gamma}+C\varepsilon^{1-2\gamma}\nonumber\\
\leq& C(\varepsilon^{1-2\gamma}+\varepsilon^{\gamma}).
\end{align}
Letting $1-2\gamma=\gamma$, we take $\gamma=1/3$. Combining (\ref{boundary1_d=2}), (\ref{boundary2_d=2}) and (\ref{boundary3_d=2}), and recalling $v_{1}^{1}-v_{1}^{*1}=0$ on $\partial\Omega$, we obtain
\begin{align*}
|(v_{1}^{1}-v_{1}^{*1})(x)|\leq C\varepsilon^{1/3},\quad x\in\partial(V\setminus \mathcal{C}_{\sqrt[3]{\varepsilon}}).
\end{align*}

Now applying the maximum principle for Lam\'{e} systems on $V\setminus \mathcal{C}_{\sqrt[3]{\varepsilon}}$ (see, e.g.\cite{mmn})  yields
\begin{align*}
|(v_{1}^{1}-v_{1}^{*1})(x)|\leq C\varepsilon^{1/3},\quad\mbox{in}~~ V\setminus \mathcal{C}_{\sqrt[3]{\varepsilon}}.
\end{align*}
Then, using the standard boundary gradient estimates for Lam\'e system (see \cite{adn}), 
\begin{align*}
|\nabla(v_{1}^{1}-v_{1}^{*1})(x)|\leq C\varepsilon^{1/3},\quad\mbox{on}~~\partial\Omega.
\end{align*}

Therefore, recalling \eqref{b10},
\begin{align}\label{b10-*}
|b_{1}^{1,0}-b_{*1}^{1,0}|=\left|\int_{\partial\Omega}\frac{\partial (v_{1}^{1}-v_{1}^{*1})}{\partial \nu}\Big|_{+}\cdot \varphi\right|\leq\,C\varepsilon^{1/3}\|\varphi\|_{L^{1}(\partial\Omega)}.
\end{align}

{\bf STEP 2.} Secondly, for $b_{1}^{1,\alpha}$, $\alpha=1,2,\cdots,d$. The proof is essentially the same. It follows from the definitions of $v^{\alpha}$, $\alpha=1,2,\cdots,d$, \eqref{equ_valpha} that
\begin{equation}\label{v_alpha}
\begin{cases}
\mathcal{L}_{\lambda,\mu}(v^{\alpha}-\psi^{\alpha})=0,&\mbox{in}~\widetilde{\Omega},\\
v^{\alpha}-\psi^{\alpha}=0,&\mbox{on}~\partial{D}_{1}\cup\partial{D}_{2},\\
v^{\alpha}-\psi^{\alpha}=-\psi^{\alpha},&\mbox{on}~\partial{\Omega}.
\end{cases}
\end{equation}
Recalling the definitions of $v_{1}^{1}$, and using the integration by parts formula (\ref{equ4-1}), we have, for $\alpha=1,2,\cdots,d$,
\begin{align*}
b_1^{1,\alpha}&=\int_{\partial{D}_{1}}\frac{\partial (v^{\alpha}-\psi^{\alpha})}{\partial \nu}\Big|_{+}\cdot v_{1}^{1}=\int_{\widetilde{\Omega}}\left(\mathbb{C}^0e(v_{1}^{1}), e(v^{\alpha}-\psi^{\alpha})\right)=\int_{\partial\Omega}\frac{\partial v_{1}^{1}}{\partial \nu}\Big|_{+}\cdot (-\psi^{\alpha}).
\end{align*}
Similarly, for $\alpha=1,2,\cdots,d$,
\begin{align*}
b_{*1}^{1,\alpha}&=\int_{\partial{D}_{1}^{*}}\frac{\partial (v^{*\alpha}-\psi^{\alpha})}{\partial \nu}\Big|_{+}\cdot\psi^{1}=\int_{\partial\Omega}\frac{\partial v_{1}^{*1}}{\partial \nu}\Big|_{+}\cdot (-\psi^{\alpha}).
\end{align*}
In view of $v_{1}^{*1}=0$ on $\partial\Omega$, by using a standard boundary estimate for elliptic system (see, \cite{adn}), it is easy to see that
\begin{equation}\label{b1alpha*}
|b_{*1}^{1,\alpha}|\leq\,C\|\psi^{\alpha}\|_{L^{1}(\partial\Omega)}=C|\partial\Omega|,\quad\alpha=1,2,\cdots,d.
\end{equation}
By applying the same argument above, we have
\begin{align}\label{b11_alpha-*}
|b_{1}^{1,\alpha}-b_{*1}^{1,\alpha}|=\left|\int_{\partial\Omega}\frac{\partial(v_{1}^{1}- v_{1}^{*1})}{\partial \nu}\Big|_{+}\cdot (-\psi^{\alpha})\right|\leq\,C\varepsilon^{1/3}|\partial\Omega|,\quad\alpha=1,2,\cdots,d.
\end{align}

{\bf STEP 3.} Finally, recalling \eqref{b11decom}, using \eqref{C*-Ci} and \eqref{maineC}, and substituting \eqref{b10-*}, \eqref{b1alpha*}, \eqref{b11_alpha-*} and \eqref{C*-Ci}, we have
\begin{align*}\label{b11-b11*}
|b_{1}^{1}-b_{*1}^{1}|=&\left|\int_{\partial{D}_{1}}\frac{\partial u^{b}}{\partial \nu}\Big|_{+}\cdot \psi^{1}-\int_{\partial{D}_{1}^{*}}\frac{\partial u^{*}}{\partial \nu}\Big|_{+}\cdot \psi^{1}\right|\nonumber\\=&\left|\left(b_{1}^{1,0}+\sum_{\alpha=1}^{d}C_{2}^{\alpha}b_{1}^{1,\alpha}\right)-\left(b_{*1}^{1,0}+\sum_{\alpha=1}^{d}C_{*}^{\alpha}b_{*1}^{1,\alpha}\right)\right|\nonumber\\
\leq&|b_{1}^{1,0}-b_{*1}^{1,0}|+\sum_{\alpha=1}^{d}|C_{2}^{\alpha}||b_{1}^{1,\alpha}-b_{*1}^{1,\alpha}|+\sum_{\alpha=1}^{d}|C_{2}^{\alpha}-C_{*}^{\alpha}||b_{*1}^{1,\alpha}|\\
\leq&\,C\max\{\varepsilon^{1/3},\rho_{d}(\varepsilon)\}\left(\|\varphi\|_{L^{1}(\partial\Omega)}+|\partial\Omega|\right).
\end{align*}
The proof of Proposition \ref{lemma5.1d=2} is completed.
\end{proof}

\begin{proof}[Proof of Lemma \ref{lem3.2}]
{\bf STEP 1. Systems of $(C_{1}^{\alpha}+C_{2}^{\alpha})/2$ and $C_{*}^{\alpha}$}.
Recalling the original decomposition \eqref{decom_u} and the forth line of \eqref{mainequation}, we have
\begin{equation}\label{C1C2_0}
\left\{
\begin{aligned}
\sum_{\alpha=1}^{d(d+1)/2}C_{1}^{\alpha}a_{11}^{\alpha\beta}+\sum_{\alpha=1}^{d(d+1)/2}C_{2}^{\alpha}a_{21}^{\alpha\beta}-\tilde{b}_{1}^{\beta}&=0,\\
\sum_{\alpha=1}^{d(d+1)/2}C_{1}^{\alpha}a_{12}^{\alpha\beta}+\sum_{\alpha=1}^{d(d+1)/2}C_{2}^{\alpha}a_{22}^{\alpha\beta}-\tilde{b}_{2}^{\beta}&=0,
\end{aligned}
\right.\quad\quad~~\beta=1,2,\cdots,\frac{d(d+1)}{2},
\end{equation}
where $$\tilde{b}_{j}^{\beta}=\int_{\partial{D}_{j}}\frac{\partial{v}_{0}}{\partial\nu}\Big|_{+}\cdot\psi^{\beta}.
$$
For the first equation of \eqref{C1C2_0},
\begin{equation*}\label{C1C2_o}
\sum_{\alpha=1}^{d(d+1)/2}(C_{1}^{\alpha}+C_{2}^{\alpha})a_{11}^{\alpha\beta}+\sum_{\alpha=1}^{d(d+1)/2}C_{2}^{\alpha}(a_{21}^{\alpha\beta}-a_{11}^{\alpha\beta})-\tilde{b}_{1}^{\beta}=0,
\end{equation*}
and
\begin{equation*}\label{C1C2_o}
\sum_{\alpha=1}^{d(d+1)/2}(C_{1}^{\alpha}+C_{2}^{\alpha})a_{21}^{\alpha\beta}+\sum_{\alpha=1}^{d(d+1)/2}C_{1}^{\alpha}(a_{11}^{\alpha\beta}-a_{21}^{\alpha\beta})-\tilde{b}_{1}^{\beta}=0.
\end{equation*}
Adding these two equations together leads to
\begin{equation}\label{C1C2_1}
\sum_{\alpha=1}^{d(d+1)/2}(C_{1}^{\alpha}+C_{2}^{\alpha})(a_{11}^{\alpha\beta}+a_{21}^{\alpha\beta})+\sum_{\alpha=1}^{d(d+1)/2}(C_{1}^{\alpha}-C_{2}^{\alpha})(a_{11}^{\alpha\beta}-a_{21}^{\alpha\beta})-2\tilde{b}_{1}^{\beta}=0.
\end{equation}
Similarly, for the second equation of \eqref{C1C2_0},
\begin{equation}\label{C1C2_22}
\sum_{\alpha=1}^{d(d+1)/2}(C_{1}^{\alpha}+C_{2}^{\alpha})(a_{12}^{\alpha\beta}+a_{22}^{\alpha\beta})+\sum_{\alpha=1}^{d(d+1)/2}(C_{1}^{\alpha}-C_{2}^{\alpha})(a_{12}^{\alpha\beta}-a_{22}^{\alpha\beta})-2\tilde{b}_{2}^{\beta}=0.
\end{equation}
A further combination of \eqref{C1C2_1} and \eqref{C1C2_22} together yields
\begin{equation}\label{C1+C2_1}
\sum_{\alpha=1}^{d(d+1)/2}\frac{C_{1}^{\alpha}+C_{2}^{\alpha}}{2}\left(\sum_{i,j=1}^{2}a_{ij}^{\alpha\beta}\right)
+\sum_{\alpha=1}^{d(d+1)/2}\frac{C_{1}^{\alpha}-C_{2}^{\alpha}}{2}(a_{11}^{\alpha\beta}-a_{22}^{\alpha\beta})-(\tilde{b}_{1}^{\beta}+\tilde{b}_{2}^{\beta})=0.
\end{equation}

Recalling that
$$u^{*}=\sum_{\alpha=1}^{d(d+1)/2}C_{*}^{\alpha}v^{*\alpha}+v_{0}^{*},$$
where $v^{*\alpha}$ and $v_{0}^{*}$ are, respectively, defined by \eqref{def_v*alpha} and \eqref{def_v0*}.
From the third line of \eqref{u00*}, we have
\begin{align}\label{equ_C*alpha}
&\sum_{\alpha=1}^{d(d+1)/2}C_{*}^{\alpha}\left(\int_{\partial{D}_{1}^{*}}\frac{\partial{v}^{*\alpha}}{\partial\nu}\big|_{+}\cdot\psi^{\beta}+\int_{\partial{D}_{2}^{*}}\frac{\partial{v}^{*\alpha}}{\partial\nu}\big|_{+}\cdot\psi^{\beta}\right)\nonumber\\
&+\left(\int_{\partial{D}_{1}^{*}}\frac{\partial{v}_{0}^{*}}{\partial\nu}\big|_{+}\cdot\psi^{\beta}+\int_{\partial{D}_{2}^{*}}\frac{\partial{v}_{0}^{*}}{\partial\nu}\big|_{+}\cdot\psi^{\beta}\right)=0,\quad\beta=1,2,\cdots,\frac{d(d+1)}{2}.
\end{align}

{\bf STEP 2. Closeness}. Next, comparing \eqref{equ_C*alpha} with \eqref{C1+C2_1}, we will prove
\begin{align}
\left|\sum_{i,j=1}^{2}a_{ij}^{\alpha\beta}-\sum_{i=1}^{2}\int_{\partial{D}_{i}^{*}}\frac{\partial{v}^{*\alpha}}{\partial\nu}\big|_{+}\cdot\psi^{\beta}\right|\leq\,C\rho_{d}(\varepsilon),\quad\alpha,\beta=1,2,\cdots,\frac{d(d+1)}{2};\label{inequ1}\end{align}
and
\begin{equation}\label{inequ2}
\left|\sum_{i=1}^{2}\tilde{b}_{i}^{\beta}-\sum_{i=1}^{2}\int_{\partial{D}_{i}^{*}}\frac{\partial{v}_{0}^{*}}{\partial\nu}\big|_{+}\cdot\psi^{\beta}\right|\leq\,C\rho_{d}(\varepsilon),\quad\beta=1,2,\cdots,\frac{d(d+1)}{2}.
\end{equation}

We only prove \eqref{inequ1} for instance. The proof of \eqref{inequ2} is the same. By the definition of $v^{\alpha}$, \eqref{v_alpha},
\begin{align*}
a^{\alpha\beta}:=\sum_{i,j=1}^{2}a_{ij}^{\alpha\beta}
&=\int_{\partial{D}_{1}}\frac{\partial{v}^{\alpha}}{\partial\nu}\big|_{+}\cdot\psi^{\beta}+\int_{\partial{D}_{2}}\frac{\partial{v}^{\alpha}}{\partial\nu}\big|_{+}\cdot\psi^{\beta}\nonumber\\
&=\int_{\partial{D}_{1}}\frac{\partial({v}^{\alpha}-\psi^{\alpha})}{\partial\nu}\big|_{+}\cdot\psi^{\beta}+\int_{\partial{D}_{2}}\frac{\partial({v}^{\alpha}-\psi^{\alpha})}{\partial\nu}\big|_{+}\cdot\psi^{\beta}\nonumber\\
&=\int_{\partial{\Omega}}\frac{\partial{v}^{\beta}}{\partial\nu}\big|_{+}\cdot(-\psi^{\alpha}).
\end{align*}
Similarly, by \eqref{def_v*alpha},
$$a^{*\alpha\beta}:=\int_{\partial{D}_{1}^{*}}\frac{\partial{v}^{*\alpha}}{\partial\nu}\big|_{+}\cdot\psi^{\beta}+\int_{\partial{D}_{2}^{*}}\frac{\partial{v}^{*\alpha}}{\partial\nu}\big|_{+}\cdot\psi^{\beta}=\int_{\partial{\Omega}}\frac{\partial{v}^{*\beta}}{\partial\nu}\big|_{+}\cdot(-\psi^{\alpha}).$$
Thus,
\begin{equation}\label{equ3}
\left|a^{\alpha\beta}-a^{*\alpha\beta}\right|=\left|\int_{\partial{\Omega}}\frac{\partial({v}^{\beta}-{v}^{*\beta})}{\partial\nu}\big|_{+}\cdot(-\psi^{\alpha})\right|.
\end{equation}

Now we use the argument in STEP 1 and STEP 2 of the proof of Proposition \ref{lemma5.1d=2} to prove that
\begin{equation}\label{difference_vbeta}
|{v}^{\beta}-{v}^{*\beta}|\leq\,C\varepsilon,\quad\mbox{in}~V,\quad\mbox{for}~\beta=1,2,\cdots,d(d+1)/2.
\end{equation}
Denote $\Gamma_{i1}:=\partial{D}_{i}^{*}\setminus\partial{D}_{i}$ and $\Gamma_{i2}:=\partial{D}_{i}\setminus\partial{D}_{i}^{*}$, $i=1,2$. Then 
$\partial{V}=\cup_{i,j=1,2}\Gamma_{ij}\cup\partial\Omega$. Setting
$$\phi^{\beta}(x):={v}^{\beta}(x)-{v}^{*\beta}(x),$$
then $\mathcal{L}_{\lambda,\mu}\phi^{\beta}=0$ in $V$. It is easy to see that $\phi^{\beta}=0$ on $\partial\Omega$. For $\beta=1,2,\cdots,d$, on $\Gamma_{11}$, by mean value theorem, and \eqref{v0bdd}, we have
\begin{align*}
|\phi^{\beta}|\Big|_{\Gamma_{11}}&=|{v}^{\beta}-{v}^{*\beta}|\Big|_{\Gamma_{11}}=|{v}^{\beta}-\psi^{\beta}|\Big|_{\Gamma_{11}}\\
&=|{v}^{\beta}(x',x_{d})-v^{\beta}(x',x_{d}+\varepsilon/2)|\Big|_{\Gamma_{11}}=|\nabla {v}^{\beta}(\xi)|\varepsilon\leq\,C\varepsilon,
\end{align*}
where $\xi\in{D}_{1}^{*}\setminus{D}_{1}$. Similarly, by \eqref{v0*bdd},
\begin{align*}
|\phi^{\beta}|\Big|_{\Gamma_{12}}&=|{v}^{\beta}-{v}^{*\beta}|\Big|_{\Gamma_{12}}=|\psi^{\beta}-{v}^{*\beta}|\Big|_{\Gamma_{12}}\\&=|{v}^{*\beta}(x',x_{d})-v^{*\beta}(x',x_{d}-\varepsilon/2)|\Big|_{\Gamma_{12}}=|\nabla {v}^{\beta}(\xi)|\varepsilon\leq\,C\varepsilon,
\end{align*}
for some $\xi\in{D}_{1}\setminus{D}_{1}^{*}$. For $\beta=d+1,\cdots,d(d+1)/2$, on $\Gamma_{11}$, 
\begin{align*}
|\phi^{\beta}|\Big|_{\Gamma_{11}}&=|{v}^{\beta}-{v}^{*\beta}|\Big|_{\Gamma_{11}}=|{v}^{\beta}-\psi^{\beta}|\Big|_{\Gamma_{11}}=|{v}^{\beta}-(\psi^{\beta}+\varepsilon)+\varepsilon|\Big|_{\Gamma_{11}}\\
&=|{v}^{\beta}(x',x_{d})-v^{\beta}(x',x_{d}+\varepsilon/2)|\Big|_{\Gamma_{11}}+\varepsilon=|\nabla {v}^{\beta}(\xi)|\varepsilon+\varepsilon\leq\,C\varepsilon,
\end{align*}
where $\xi\in{D}_{1}^{*}\setminus{D}_{1}$. On $\Gamma_{12}$ is the same. By the same way,
$$|\phi^{\beta}|\Big|_{\Gamma_{21}\cup\Gamma_{22}}\leq\,C\varepsilon,\quad\mbox{for}~\beta=1,2,\cdots,d(d+1)/2.$$
Applying the maximum principle to $\phi^{\beta}$ on $V$ (\cite{mmn}), yields \eqref{difference_vbeta}. 

Then, using the standard boundary gradient estimates for Lam\'e system again, 
\begin{align*}
|\nabla\phi^{\beta}|\leq C\varepsilon,\quad\mbox{on}~~\partial\Omega.
\end{align*}
Therefore, recalling \eqref{equ3},
\begin{align*}\label{b10-*}
\left|a^{\alpha\beta}-a^{*\alpha\beta}\right|\leq\,C\varepsilon\|\psi^{\alpha}\|_{L^{1}(\partial\Omega)},\quad\alpha,\beta=1,2,\cdots,\frac{d(d+1)}{2}.
\end{align*}

{\bf STEP 3. Invertibility of the coefficients matrix $a^{*\alpha\beta}$}. On the other hand,
$$a^{*\alpha\beta}=\int_{\partial{D}_{1}^{*}}\frac{\partial{v}^{*\alpha}}{\partial\nu}\big|_{+}\cdot\psi^{\beta}+\int_{\partial{D}_{2}^{*}}\frac{\partial{v}^{*\alpha}}{\partial\nu}\big|_{+}\cdot\psi^{\beta}=\int_{\widetilde{\Omega}^{*}}\left(\mathbb{C}^0e(v^{*\alpha}), e(v^{*\beta})\right).$$
We claim that $a^{*\alpha\beta}$ is positive definite, so invertible. Moveover, there exists a universal constant $C$ such that
\begin{equation}\label{a_positive}
\sum_{\alpha,\beta=1}^{d(d+1)/2}a^{*\alpha\beta}\xi^{\alpha}\xi^{\beta}\geq\frac{1}{C},\quad\forall~|\xi|=1.
\end{equation}
Indeed, if $e\Big(\sum_{\alpha=1}^{d(d+1)/2}\xi_{\alpha}v^{*\alpha}\Big)=0$ in $\widetilde{\Omega}^{*}$, then $\sum_{\alpha=1}^{d(d+1)/2}\xi_{\alpha}v^{*\alpha}=\sum_{\alpha=1}^{d(d+1)/2}a_{\alpha}\psi^{\alpha}$ in $\widetilde{\Omega}^{*}$, for some constants $a_{\alpha}$. Since $\sum_{\alpha=1}^{d(d+1)/2}\xi_{\alpha}v^{*\alpha}\big|_{\partial{\Omega}}=0$, it follows that $a_{1}=a_{2}=\cdots=a_{d(d+1)/2}=0$. Hence, $|\xi|=0$ by using $v^{*\alpha}=\psi^{\alpha}$ on $\partial{D}_{i}^{*}$.  This is a contradiction. 

{\bf STEP 4. Completion}. Finally, we notice from $\bar{u}=1-\underline{u}$ in $\Omega_{R}$ that $\nabla\bar{u}_{1}^{\alpha}=-\nabla\bar{u}_{2}^{\alpha}$ in $\Omega_{R}$, $\alpha=1,2,\cdots,d(d+1)/2$. Then making use of \eqref{nabla_w_ialpha} and \eqref{nabla_w_out},
\begin{align*}
\left|a_{11}^{\alpha\beta}-a_{22}^{\alpha\beta}\right|=&\,\Big|\int_{\Omega_{R}}\left(\mathbb{C}^0e(v_{1}^{\alpha}), e(v_{1}^{\beta})\right)-\int_{\Omega_{R}}\left(\mathbb{C}^0e(v_{2}^{\alpha}), e(v_{2}^{\beta})\right)\\
&+\int_{\widetilde{\Omega}\setminus\Omega_{R}}\left(\mathbb{C}^0e(v_{1}^{\alpha}), e(v_{1}^{\beta})\right)-\int_{\widetilde{\Omega}\setminus\Omega_{R}}\left(\mathbb{C}^0e(v_{2}^{\alpha}), e(v_{2}^{\beta})\right)\Big|\\
\leq&\,\Big|\int_{\Omega_{R}}\left(\mathbb{C}^0e(\bar{u}_{1}^{\alpha}), e(\bar{u}_{1}^{\beta})\right)-\int_{\Omega_{R}}\left(\mathbb{C}^0e(\bar{u}_{2}^{\alpha}), e(\bar{u}_{2}^{\beta})\right)\Big|\\
&+\Big|\int_{\Omega_{R}}\left(\mathbb{C}^0e(v_{1}^{\alpha}-\bar{u}_{1}^{\alpha}), e(v_{1}^{\beta})\right)-\int_{\Omega_{R}}\left(\mathbb{C}^0e(v_{2}^{\alpha}-\bar{u}_{2}^{\alpha}), e(v_{2}^{\beta})\right)\Big|+C\\
\leq&\,C.
\end{align*}
So that, by \eqref{maineC1-C2},
\begin{equation}\label{equ_x}
\left|\frac{C_{1}^{\alpha}-C_{2}^{\alpha}}{2}(a_{11}^{\alpha\beta}-a_{22}^{\alpha\beta})\right|\leq\,C\rho_{d}(\varepsilon),\quad\mbox{for}~~\alpha=1,2,\cdots,d.
\end{equation}

It follows from \eqref{a_positive} and \eqref{inequ1} that for sufficiently small $\varepsilon$, $(a^{\alpha\beta})$ is also invertible. So that for sufficiently small $\varepsilon$, in view of \eqref{inequ2}, and \eqref{equ_x} for $\alpha=1,2,\cdots,d$, it follows from comparing \eqref{C1+C2_1} and \eqref{equ_C*alpha} that the proof of Lemma \ref{lem3.2} is finished.
\end{proof}

\bigskip

\section{Proof of Theorem \ref{thm1.2} in dimension two}\label{sec4}

In this section, we first give an improvement of estimates for $|C_{1}^{\alpha}-C_{2}^{\alpha}|$, $\alpha=1,2$, especially including a lower bound, which contains a non-zero factor $b_{1}^{\alpha}$ ($b_{*1}^{\alpha}$ is its limit). This is due to a careful selection from the whole system of $C_{i}^{\alpha}, $\eqref{C1C2_3}, although it seems a little different with that in \cite{bll}. From it we can see the role of the blow-up factor $b_{*1}^{\alpha}$ in such singularity analysis of $|\nabla u|$. 

\begin{prop}\label{prop4.1}If $b_{1}^{\alpha}\neq0$, then
\begin{equation}\label{C1-C2_d=2}
\frac{\sqrt{\varepsilon}}{C}\Big|b_{1}^{\alpha}[\varphi]\Big|+o(\sqrt{\varepsilon})\leq\big|C_{1}^{\alpha}-C_{2}^{\alpha}\big|\leq\,C\sqrt{\varepsilon},\quad\alpha=1,2.
\end{equation}
\end{prop}

In order to solve $C_{1}^{1}-C_{2}^{1}$ and $C_{1}^{2}-C_{2}^{2}$ from \eqref{C1C2_3}, we choose $\beta=1,2,3$, for $j=1$, and $\beta=3$ for $j=2$. Then
$$
AX=\begin{pmatrix}
~a_{11}^{11}&a_{11}^{12}&a_{11}^{13}&a_{12}^{13}~\\\\
~a_{11}^{21}&a_{11}^{22}&a_{11}^{23}&a_{12}^{23}~\\\\
~a_{11}^{31}&a_{11}^{32}&a_{11}^{33}&a_{12}^{33}~\\\\
~a_{21}^{31}&a_{21}^{32}&a_{21}^{33}&a_{22}^{33}
\end{pmatrix}\begin{pmatrix}
~C_{1}^{1}-C_{2}^{1}~\\\\
C_{1}^{2}-C_{2}^{2}\\\\
C_{1}^{3}\\\\
C_{2}^{3}
\end{pmatrix}
=\begin{pmatrix}
~b_{1}^{1}~\\\\
b_{1}^{2}\\\\
b_{1}^{3}\\\\
b_{2}^{3}
\end{pmatrix}.
$$

\subsection{Refined estimates in dimension $d=2$ }

We first give the following refined estimates for $a_{ij}^{\alpha\beta}$.

\begin{lemma}\label{lem_a_11d=2}
$A$ is positive definite, and
\begin{equation}\label{a11_11}
\frac{1}{C\sqrt{\varepsilon}}\leq\,a_{11}^{\alpha\alpha}\leq\,\frac{C}{\sqrt{\varepsilon}},\qquad\alpha=1,2;
\end{equation}
\begin{equation}\label{a11_12d=2}
|a_{11}^{12}|=|a_{11}^{21}|\leq\,C|\log\varepsilon|;
\end{equation}
\begin{equation}\label{a11_33}
\frac{1}{C}\leq\,a_{ii}^{33}\leq\,C,\quad\,i=1,2;
\end{equation}

\begin{equation}\label{a11_13}
\left|a_{12}^{33}\right|=\left|a_{21}^{33}\right|,\left|a_{ij}^{\alpha3}\right|=\left|a_{ji}^{3\alpha}\right|\leq\,C,\quad\,i=1,2,\alpha=1,2;
\end{equation}
and
\begin{equation}\label{bjbeta}
\left|b_{j}^{\beta}\right|\leq\,C.
\end{equation}
\end{lemma}

\begin{remark}
    Estimates \eqref{a11_11}, \eqref{a11_33}, \eqref{a11_13} are the same as in \cite{bll}, we omit their proof here. Estimate \eqref{a11_12d=2} is an improvement of (4.13) in \cite{bll}, $|a_{11}^{12}|=|a_{11}^{21}|\leq\,C\varepsilon^{-1/4}$.  While \eqref{bjbeta} needs to be shown since the definition of $b_{j}^{\beta}$ is different with that in \cite{bll}. So we only prove \eqref{a11_12d=2} and \eqref{bjbeta} below.
\end{remark}

\begin{proof}[Proof of Lemma \ref{lem_a_11d=2}]
For \eqref{a11_12d=2}, we take $a_{11}^{12}$ for instance. By definition \eqref{def_bj},
\begin{align*}
a_{11}^{12}=-\int_{\partial{D}_{1}}\frac{\partial{v}_{1}^{1}}{\partial\nu}\Big|_{+}\cdot\psi^{2}&=-\int_{\partial{D}_{1}}\frac{\partial\bar{u}_{1}^{1}}{\partial\nu}\Big|_{+}\cdot\psi^{2}-\int_{\partial{D}_{1}}\frac{\partial({v}_{1}^{1}-\bar{u}_{1}^{1})}{\partial\nu}\Big|_{+}\cdot\psi^{2}\\
&:=-\mathrm{I}-\mathrm{II},
\end{align*}
where
$$\mathrm{I}=\int_{\partial{D}_{1}}\frac{\partial\bar{u}_{1}^{1}}{\partial\nu}\Big|_{+}\cdot\psi^{2}=\int_{\partial{D}_{1}\cap\mathcal{C}_{R}}\frac{\partial\bar{u}_{1}^{1}}{\partial\nu}\Big|_{+}\cdot\psi^{2}+\int_{\partial{D}_{1}\setminus\mathcal{C}_{R}}\frac{\partial\bar{u}_{1}^{1}}{\partial\nu}\Big|_{+}\cdot\psi^{2}:=\mathrm{I}_{R}+O(1),$$
and by using \eqref{nabla_w_ialpha},
$$|\mathrm{II}|\leq\,C.$$

On boundary $\partial{D}_{1}$,
$$n_{1}=\frac{\partial_{x_{1}}h_{1}(x_{1})}{\sqrt{1+|\partial_{x_{1}}h_{1}(x_{1})|^{2}}},\quad\,n_{2}=\frac{1}{\sqrt{1+|\partial_{x_{1}}h_{1}(x_{1})|^{2}}}.$$
Recalling
$$\bar{u}_{1}^{1}=\begin{pmatrix}
~\bar{u}~\\
0
\end{pmatrix},\quad\nabla\bar{u}_{1}^{1}=\begin{pmatrix}
~~\partial_{x_{1}}\bar{u}&\partial_{x_{2}}\bar{u}~~\\
0&0
\end{pmatrix},$$
then
$$\left(\nabla\bar{u}_{1}^{1}+(\nabla\bar{u}_{1}^{1})^{T}\right)\vec{n}=\begin{pmatrix}
~~2\partial_{x_{1}}\bar{u}n_{1}+\partial_{x_{2}}\bar{u}n_{2}~~\\
\partial_{x_{2}}\bar{u}n_{1}
\end{pmatrix}.$$
Thus,
\begin{align*}
\mathrm{I}_{R}=&\int_{\partial{D}_{1}\cap\mathcal{C}_{R}}\frac{\partial\bar{u}_{1}^{1}}{\partial\nu}\Big|_{+}\cdot\psi^{2}\\=&\int_{\partial{D}_{1}\cap\mathcal{C}_{R}}\left(\lambda\left(\nabla\cdot\bar{u}_{1}^{1}\right)\vec{n}+\mu\left(\nabla\bar{u}_{1}^{1}+(\nabla\bar{u}_{1}^{1})^{T}\right)\vec{n}\right)\cdot\begin{pmatrix}
~0~\\
1
\end{pmatrix}dS\\
=&\int_{\partial{D}_{1}\cap\mathcal{C}_{R}}\lambda\left(\partial_{x_{1}}\bar{u}\right)n_{2}+\mu\partial_{x_{2}}\bar{u}n_{1}dS.
\end{align*}
So that
\begin{align*}
|\mathrm{I}_{R}|\leq&\Big|\int_{\partial{D}_{1}\cap\mathcal{C}_{R}}\lambda\left(\partial_{x_{1}}\bar{u}\right)n_{2}+\mu\partial_{x_{2}}\bar{u}n_{1}dS\Big|\\
\leq&\,C\int_{|x_{1}|\leq\,R}\frac{|x_{1}|}{\varepsilon+|x_{1}|^{2}}dx_{1}\\
\leq&\,C|\log\varepsilon|.
\end{align*}
Therefore
$$|a_{11}^{12}|\leq|\mathrm{I}|+|\mathrm{II}|\leq\,C|\log\varepsilon|.$$

By definition \eqref{def_bj}, and ${u}_{B{u}_{B}}=\sum_{\alpha=1}^{2}C_{2}^{\alpha}v^{\alpha}+v_{0}$, we have
\begin{align*}
b_{1}^{\beta}&=\int_{\partial{D}_{1}}\frac{\partial{u}_{b}}{\partial\nu}\Big|_{+}\cdot\psi^{\beta}\\
&=C_{2}^{1}\int_{\partial{D}_{1}}\frac{\partial{v}^{1}}{\partial\nu}\Big|_{+}\cdot\psi^{\beta}+C_{2}^{2}\int_{\partial{D}_{1}}\frac{\partial{v}^{2}}{\partial\nu}\Big|_{+}\cdot\psi^{\beta}+\int_{\partial{D}_{1}}\frac{\partial{v}^{0}}{\partial\nu}\Big|_{+}\cdot\psi^{\beta}\\
&=C_{2}^{1}\mathrm{I}_{1}+C_{2}^{2}\mathrm{I}_{2}+\mathrm{I}_{0}.
\end{align*}
By using integration by parts, \eqref{equ4-1}, and \eqref{mainev0},
$$|\mathrm{I}_{0}|=|\int_{\partial{D}_{1}}\frac{\partial{v}^{0}}{\partial\nu}\Big|_{+}\cdot\psi^{\beta}|=|\int_{\widetilde{\Omega}}\Big(\mathbb{C}e(v_{0},e(v_{1}^{\beta}))\Big)|\leq\,C.$$
Similarly, by using \eqref{mainev1+v2}, we have
$$|\mathrm{I}_{1}|,|\mathrm{I}_{2}|\leq\,C.$$
Combining these with \eqref{maineC}, $|C_{i}^{\alpha}|\leq\,C$,  the proof of \eqref{bjbeta} is finished.

Finally, we prove that $A$ is positive definite. For $\xi=(\xi_{1},\xi_{2},\xi_{3},\xi_{4})^{T}\neq0$, by elliptic condition, \eqref{coeff3_convex}, we have
\begin{align*}
\xi^{T}
A\xi&=\int_{\widetilde{\Omega}}\Big(\mathbb{C}e(\xi_{1}v_{1}^{1}+\xi_{2}v_{1}^{2}+\xi_{3}v_{1}^{3}+\xi_{4}v_{2}^{3}),e(\xi_{1}v_{1}^{1}+\xi_{2}v_{1}^{2}+\xi_{3}v_{1}^{3}+\xi_{4}v_{2}^{3})\Big)\\
&\geq\int_{\widetilde{\Omega}}|e(\xi_{1}v_{1}^{1}+\xi_{2}v_{2}^{2}+\xi_{3}v_{1}^{3}+\xi_{4}v_{2}^{3})|^{2}>0.
\end{align*}
In the last inequality, we used the fact that $e(\xi_{1}v_{1}^{1}+\xi_{2}v_{1}^{2}+\xi_{3}v_{1}^{3}+\xi_{4}v_{2}^{3})$ is not identically zero. Indeed, if an element $\xi\in\Psi$ vanishes at two distinct  points  $\bar{x}^{1},\bar{x}^{2}$, then $\xi\equiv0$, see lemma 6.1 in \cite{bll2}. Namely, if $e(\xi_{1}v_{1}^{1}+\xi_{2}v_{1}^{2}+\xi_{3}v_{1}^{3}+\xi_{4}v_{2}^{3})=0$, then $\xi_{1}v_{1}^{1}+\xi_{2}v_{1}^{2}+\xi_{3}v_{1}^{3}+\xi_{4}v_{2}^{3}=\sum_{\alpha=1}^{3}a_{\alpha}\psi^{\alpha}$ for some constants $a_{\alpha}$. Since $\xi_{1}v_{1}^{1}+\xi_{2}v_{1}^{2}+\xi_{3}v_{1}^{3}+\xi_{4}v_{2}^{3}\big|_{\partial{\Omega}}=0$, and $\{\psi^{\alpha}\}_{\alpha=1,2,3}$ is linear independent, it follows that $a_{1}=a_{2}=a_{3}=0$. Since $\xi_{1}v_{1}^{1}+\xi_{2}v_{1}^{2}+\xi_{3}v_{1}^{3}+\xi_{4}v_{2}^{3}\big|_{\partial{D}_{2}}=\xi_{4}\psi^{3}$, so that $\xi_{4}=0$. While, by $\xi_{1}v_{1}^{1}+\xi_{2}v_{1}^{2}+\xi_{3}v_{1}^{3}+\xi_{4}v_{2}^{3}\big|_{\partial{D}_{1}}=\xi_{1}\psi^{1}+\xi_{2}\psi^{2}+\xi_{3}\psi^{3}$ and the same reason, $\xi_{1}=\xi_{2}=\xi_{3}=0$.  This is a contradiction. 
\end{proof}

From the fact that $A$ is positive definite, we know that its principle minor $A_{22}=\begin{pmatrix}
a_{11}^{33}&a_{12}^{33}\\
a_{21}^{33}&a_{22}^{33}
\end{pmatrix}$ is also positive definite. Furthermore, we have

\begin{lemma}\label{lem4.4}
There is a universal constant $C$, independent of $\varepsilon$, such that
$$\det A_{22}=a_{11}^{33}a_{22}^{33}-a_{12}^{33}a_{21}^{33}>\frac{1}{C}.$$
\end{lemma}

\begin{proof}
From elliptic condition, \eqref{coeff3_convex}, it suffices to prove that 
$$\int_{\widetilde{\Omega}}|e(\xi_{1}v_{1}^{3}+\xi_{2}v_{2}^{3})|^{2}>\frac{1}{C},\quad\forall~\xi=(\xi_{1},\xi_{2})^{T}, ~|\xi|=1.$$
Indeed, if not, then there exist a sequence $\varepsilon_{k}\rightarrow0^{+}$, $|\xi^{k}|=1$, such that
\begin{equation}\label{convergence}
\int_{\widetilde{\Omega}}|e(\xi_{1}^{k}v_{1}^{3,\varepsilon_{k}}+\xi_{2}^{k}v_{2}^{3,\varepsilon_{k}})|^{2}\rightarrow0,\quad\mbox{as}~i\rightarrow\infty.
\end{equation}
Here we add superscript $\varepsilon_{k}$ to denote the solution of \eqref{v1alpha} when $\mathrm{dist}(D_{1},D_{2})=\varepsilon_{k}$. Since $v_{i}^{3,\varepsilon_{k}}\equiv0$ on $\partial\Omega$, it follows from the second KornÕs inequality  (see Theorem 2.5 in \cite{osy}) that there exists a constant C, independent of $\varepsilon_{k}$, such that
$$\|v_{i}^{3,\varepsilon_{k}}\|_{H^{1}(\widetilde\Omega\setminus{B}_{\bar{r}};\mathbb{R}^{d})}\leq\,C,\quad\,i=1,2,$$
for some $\bar{r}>0$ (say, $\bar{r}=R/2$). Then there exists a subsequence, still denote $\{v_{i}^{3,\varepsilon_{k}}\}$, such that
$$v_{i}^{3,\varepsilon_{k}}\rightharpoonup\,v_{i}^{*3},\quad\mbox{in}~H^{1}(\widetilde{\Omega}^{*}\setminus{B}_{\bar{r}};\mathbb{R}^{d}),\quad\mbox{as}~k\rightarrow+\infty,\quad\,i=1,2.$$
It follows from \eqref{convergence} that there exists $\xi^{*}$ such that
$$\xi^{k}\rightarrow\xi^{*},\quad\mbox{as}~k\rightarrow+\infty,\quad\mbox{with}~|\xi^{*}|=1,$$
and
$$\int_{\widetilde{\Omega}^{*}\setminus{B}_{\bar{r}}}|e(\xi_{1}^{*}v_{1}^{*3}+\xi_{2}^{*}v_{2}^{*3})|^{2}=0,$$
where $v_{i}^{*3}$, $i=1,2$, are defined by
\begin{align}\label{vi_alpha*}
\begin{cases}
  \mathcal{L}_{\lambda,\mu}v_{i}^{*\alpha}=0,\quad&
\hbox{in}\  \widetilde\Omega^{*},  \\
v_{i}^{*\alpha}=\psi^{\alpha},\ &\hbox{on}\ \partial{D}_{i}^{*}\setminus\{0\},\\
v_{i}^{*\alpha}=0,&\hbox{on} \ \partial{D}_{j}^{*}\cup\partial\Omega,~j\neq i,
\end{cases}
\end{align}
with $\alpha=3$. This implies that
$$e(\xi_{1}^{*}v_{1}^{*3}+\xi_{2}^{*}v_{2}^{*3})=0,\quad\mbox{in}~\widetilde{\Omega}^{*}\setminus{B}_{\bar{r}}.$$
Hence, 
$$\xi_{1}^{*}v_{1}^{*3}+\xi_{2}^{*}v_{2}^{*3}=\sum_{\alpha=1}^{3}a_{\alpha}\psi^{\alpha},\quad\mbox{in}~\widetilde{\Omega}^{*}\setminus{B}_{\bar{r}},$$
for some constants $a_{\alpha}$, $\alpha=1,2,3$. 

Since $\xi_{1}^{*}v_{1}^{*3}+\xi_{2}^{*}v_{2}^{*3}\Big|_{\partial{\Omega}}=0$, we have $a_{1}=a_{2}=a_{3}=0$. Indeed, if an element $v\in\Psi$ vanishes at two distinct  points  $\bar{x}^{1},\bar{x}^{2}$, then $v\equiv0$, see lemma 6.1 in \cite{bll2}. Namely, restricting on one part of $\partial{D}_{2}^{*}$, we have $\xi_{2}^{*}=0$. Restricting on one part of $\partial{D}_{1}^{*}$, we have $\xi_{1}^{*}=0$. This is a contradiction with $|\xi^{*}|=1$. The proof is finished.
\end{proof}

\begin{proof}[Proof of Proposition \ref{prop4.1}]
By the estimates in Lemma \ref{lem_a_11d=2}, it follows that $A$ is positive definite, and
$$\frac{1}{C\varepsilon}\leq\det A\leq\,\frac{C}{\varepsilon}.$$
So that $A$ is invertible, and 
$$C_{1}^{1}-C_{2}^{1}=\frac{b_{1}^{1}a_{11}^{22}(a_{11}^{33}a_{22}^{33}-a_{12}^{33}a_{21}^{33})}{\det A}+o(\sqrt{\varepsilon}),$$
and
$$C_{1}^{2}-C_{2}^{2}=\frac{b_{1}^{2}a_{11}^{11}(a_{11}^{33}a_{22}^{33}-a_{12}^{33}a_{21}^{33})}{\det A}+o(\sqrt{\varepsilon}).$$
On one hand, it is easy to obtain from Lemma \ref{lem_a_11d=2} again that the upper bound
$$|C_{1}^{\alpha}-C_{2}^{\alpha}|\leq\,C\sqrt{\varepsilon},\quad\alpha=1,2.$$
On the other hand, since $a_{11}^{33}a_{22}^{33}-a_{12}^{33}a_{21}^{33}\geq\frac{1}{C}$, then if $b_{1}^{\alpha}\neq0$, then
$$|C_{1}^{\alpha}-C_{2}^{\alpha}|=\frac{\sqrt{\varepsilon}}{C}|b_{1}^{\alpha}|+o(\sqrt{\varepsilon}).$$
Thus, Proposition \ref{prop4.1} is proved.
\end{proof}

Finally, we give the proof of Theorem \ref{thm1.2}.

\begin{proof}[Proof of Theorem \ref{thm1.2} in dimension two] By Proposition \ref{lemma5.1d=2}, we have
$$b_{1}^{\alpha}-b_{*1}^{\alpha}=O(\varepsilon^{1/3}),\quad\alpha=1,2.$$
Then if there is an $k_{0}\in\{1,2\}$ such that $b_{*1}^{k_{0}}\neq0$, then for sufficiently small $\varepsilon$, 
$$\left|\nabla{u}(x)\right|
\geq\left|\sum_{\alpha=1}^{2}(C_{1}^{\alpha}-C_{2}^{\alpha})\nabla{v}_{1}^{\alpha}(x)\right|
-C\geq\frac{|b_{1}^{k_{0}}|}{C\sqrt{\varepsilon}}\geq\frac{|b_{*1}^{k_{0}}+O(\varepsilon^{1/3})|}{C\sqrt{\varepsilon}}\geq\frac{|b_{*1}^{k_{0}}|}{C\sqrt{\varepsilon}}.
$$
The proof is finished.
\end{proof}

\bigskip

\section{Proof of Theorem \ref{thm1.2} in dimension three}\label{sec5}

In this section, we are devoted to proving Theorem \ref{thm1.2} in dimension three. Similarly, as in Section \ref{sec4}, we first give lower and upper bounds of estimates for $|C_{1}^{\alpha}-C_{2}^{\alpha}|$, $\alpha=1,2,3$. Here the selection from the whole system for $C_{i}^{\alpha}$, \eqref{C1C2_3}, is different from that in \cite{bll2}. 

\begin{prop}\label{prop5.1}If $b_{1}^{\alpha}[\varphi]\neq0$, then
\begin{equation}\label{C1-C2_d=3}
\frac{1}{C|\log\varepsilon|}\Big|b_{1}^{\alpha}[\varphi]\Big|+O(|\log\varepsilon|^{-2})\leq\big|C_{1}^{\alpha}-C_{2}^{\alpha}\big|\leq\,\frac{C}{|\log\varepsilon|},\quad\alpha=1,2.
\end{equation}
\end{prop}

\subsection{Finer Estimates in Dimension $d=3$}

In order to solve $C_{1}^{1}-C_{2}^{1}$, $C_{1}^{2}-C_{2}^{2}$,  and $C_{1}^{3}-C_{2}^{3}$ from \eqref{C1C2_3}, we take $\beta=1,2,\cdots,6$ for $j=1$ and $\beta=4,5,6,$ for $j=2$. Then
$$
AX=\begin{pmatrix}
A_{11} & A_{12}\\
A_{21}& A_{22}
\end{pmatrix}
\begin{pmatrix}
X_{1}\\
X_{2}
\end{pmatrix}=
\begin{pmatrix}
B_{1}\\
B_{2}
\end{pmatrix},
$$
where 
$$A_{11}=\begin{pmatrix}
      a^{11}_{11} & a_{11}^{12} & a_{11}^{13} ~\\\\
      a_{11}^{21} & a_{11}^{22}  & a_{11}^{23} ~\\\\
      a_{11}^{31} & a_{11}^{32}  & a_{11}^{33} ~
    \end{pmatrix},
\quad
A_{12}=\begin{pmatrix}
      a_{11}^{\alpha\beta},a_{12}^{\alpha\beta}
    \end{pmatrix}_{\alpha=1,2,3;\beta=4,5,6},\quad
$$
$$ A_{21}=\begin{pmatrix}
     ~ a_{11}^{\alpha\beta}~ \\\\
      a_{21}^{\alpha\beta}
    \end{pmatrix}_{\alpha=4,5,6;\beta=1,2,3},\quad
 A_{22}=\begin{pmatrix}
     ~ a_{11}^{\alpha\beta} &a_{12}^{\alpha\beta}~\\\\
      a_{21}^{\alpha\beta}&a_{22}^{\alpha\beta}
    \end{pmatrix}_{\alpha,\beta=4,5,6}.
$$

$$X_{1}=\Big(C_{1}^{1}-C_{2}^{1}, C_{1}^{2}-C_{2}^{2}, C_{1}^{3}-C_{2}^{3}\Big)^{T},\quad
X_{2}=\Big(C_{1}^{4},C_{1}^{5},C_{1}^{6},C_{2}^{4},C_{2}^{5},C_{2}^{6}\Big)^{T},$$
and
$$B_{1}=\Big(b_{1}^{1},b_{1}^{2}, b_{1}^{3}\Big)^{T},\quad
B_{2}=\Big(b_{1}^{4},b_{1}^{5}, b_{1}^{6},b_{2}^{4},b_{2}^{5}, b_{2}^{6}\Big)^{T}.$$

\begin{lemma}(\cite{bll2})\label{lem_a_11}
$A$ is positive definite, and
\begin{equation*}\label{a11_1122}
\frac{|\log\varepsilon|}{C}\leq\,a_{11}^{\alpha\alpha}\leq\,C|\log\varepsilon|,\qquad\alpha=1,2,3;
\end{equation*}
\begin{equation*}\label{a11_456}
\frac{1}{C}\leq\,a_{ii}^{\alpha\alpha}\leq\,C,\quad\alpha=4,5,6,\quad\,i=1,2;
\end{equation*}
and
\begin{equation*}\label{a11_12}
\left|a_{ij}^{\alpha\beta}\right|=\left|a_{ji}^{\beta\alpha}\right|\leq\,C,\qquad\alpha,\beta=1,2,\cdots,\frac{d(d+1)}{2},~\alpha\neq\beta.
\end{equation*}
$$|b_{\beta}|\leq C, \quad\beta=1,2,\cdots,\frac{d(d+1)}{2}.$$
\end{lemma}
Therefore,
\begin{equation*}\label{a11_det}
\frac{|\log\varepsilon|^{3}}{C}\leq\det{A_{11}}\leq\,C\,|\log\varepsilon|^{3}.
\end{equation*}
\begin{proof}
The estimate of $b_{1}^{\beta}$ can be proved by a very similar way as in the proof of Lemma \ref{lem_a_11d=2}. We omit it here.
\end{proof}

\begin{lemma}\label{lem5.2}
There is a universal constant $C$ such that, for any $\xi=(\xi^{1},\xi^{2},\cdots,\xi^{6})^{T}\neq0$,
$$\xi^{T}A_{22}\xi>\frac{1}{C}|\xi|^{2}.$$
\end{lemma}

\begin{proof}The proof is similar to that of Lemma \ref{lem4.4}. We omit the limit process, since it is the same. After it, if there exists a vector
$\xi^{*}=(\xi^{*}_{1},\xi^{*}_{2},\cdots,\xi^{*}_{6})^{T}$ with $|\xi^{*}|=1$, such that 
\begin{align*}
e(\sum_{\alpha=1}^{3}\xi^{*}_{\alpha}v_{1}^{*(\alpha+3)}+\sum_{\alpha=4}^{6}\xi^{*}_{\alpha}v_{2}^{*\alpha})=0,\quad\mbox{in}~\widetilde{\Omega}^{*}\setminus{B}_{\bar{r}},
\end{align*}
where $v_{i}^{*\beta}$ are defined in \eqref{vi_alpha*} for $\beta=4,5,6$.
Indeed, by using lemma 6.1 in \cite{bll2} again, if an element $v\in\Psi$ vanishes at three distinct  points  $\bar{x}^{1},\bar{x}^{2}$, and $\bar{x}^{3}$, which are not on a plane, then $v\equiv0$. Namely, there exist $a_{\alpha}$, $\alpha=1,2,\cdots,6$, such that
$$\sum_{\alpha=1}^{3}\xi^{*}_{\alpha}v_{1}^{*(\alpha+3)}+\sum_{\alpha=4}^{6}\xi^{*}_{\alpha}v_{2}^{*\alpha}=\sum_{\alpha=1}^{6}a_{\alpha}\psi^{\alpha}.$$
Recalling that $v_{i}^{*\alpha}=0$ on $\partial\Omega$, we have $a_{\alpha}\equiv0$, $\alpha=1,2,\cdots,6$. Restricting on one part of $\partial{D}_{1}^{*}$, in view of the linear independence of $\psi^{4},\psi^{5}$ and $\psi^{6}$ on $\partial{D}_{1}^{*}$, we have $\xi^{*}_{1}=\xi^{*}_{2}=\xi^{*}_{3}=0$. By the same reason on $\partial{D}_{2}$, $\xi^{*}_{4}=\xi^{*}_{5}=\xi^{*}_{6}=0$. This is a contradiction with $|\xi^{*}|=1$. 
\end{proof}

\begin{proof}[Proof of Proposition \ref{prop5.1}]
By the estimates in Lemma \ref{lem_a_11}, it follows that 
$$\frac{|\log\varepsilon|^{3}}{C}\leq\det A\leq\,C|\log\varepsilon|^{3},$$
and $A$ is invertible, then by Cramer's rule,
$$C_{1}^{1}-C_{2}^{1}=\frac{b_{1}^{1}a_{11}^{22}a_{11}^{33}\det A_{22}}{\det A}+O(|\log\varepsilon|^{-2}),$$
$$C_{1}^{2}-C_{2}^{2}=\frac{b_{1}^{2}a_{11}^{11}a_{11}^{33}\det A_{22}}{\det A}+O(|\log\varepsilon|^{-2}),$$
and
$$C_{1}^{3}-C_{2}^{3}=\frac{b_{1}^{3}a_{11}^{11}a_{11}^{22}\det A_{22}}{\det A}+O(|\log\varepsilon|^{-2}).$$

On one hand, it is easy to obtain from Lemma \ref{lem_a_11} again that the upper bound
$$|C_{1}^{\alpha}-C_{2}^{\alpha}|\leq\,\frac{C}{|\log\varepsilon|},\quad\alpha=1,2,3.$$
On the other hand, since $\det A_{22}\geq\frac{1}{C}$, then if $b_{1}^{\alpha}\neq0$, then
$$|C_{1}^{\alpha}-C_{2}^{\alpha}|\geq\frac{|b_{1}^{\alpha}|}{C|\log\varepsilon|}+O(|\log\varepsilon|^{-2}).$$
Thus, Proposition \ref{prop5.1} is proved.
\end{proof}

\begin{proof}[Proof of Theorem \ref{thm1.2} in dimension three] For $d=3$, using the fact that $|\nabla{v}_{i}^{\beta}(0',x_{d})|=0$ for $|x_{d}|<\varepsilon/2$, $\beta=4,5,6$, and estimates \eqref{mainev0}, \eqref{mainev1+v2}, \eqref{mainevi3}, and \eqref{C1-C2_d=3}, we have for $x=(0',x_{d})\in\widetilde{\Omega}$,
\begin{align}\label{nablaud=3}
|\nabla u(x)|=&\,\Big|~\sum_{\alpha=1}^{3}\left(C_{1}^{\alpha}-C_{2}^{\alpha}\right)\nabla{v}_{1}^{\alpha}+\sum_{\alpha=1}^{3}C_{2}^{\alpha}\nabla(v_{1}^{\alpha}+v_{2}^{\alpha})
+\sum_{i=1}^{2}\sum_{\alpha=4}^{6}C_{i}^{\alpha}\nabla{v}_{i}^{\alpha}+\nabla{v}_{0}~\Big|(x)\nonumber\\
\geq&\,\Big|~\sum_{\alpha=1}^{3}\left(C_{1}^{\alpha}-C_{2}^{\alpha}\right)\nabla{v}_{1}^{\alpha}(x)
~\Big|-\left(\sum_{\alpha=1}^{3}|C_{2}^{\alpha}|\,|\nabla(v_{1}^{\alpha}+v_{2}^{\alpha})(x)|
+|\nabla{v}_{0}(x)|+C\right)\nonumber\\
\geq&\,\Big|~\sum_{\alpha=1}^{3}\left(C_{1}^{\alpha}-C_{2}^{\alpha}\right)\nabla{v}_{1}^{\alpha}(x)
~\Big|-C\nonumber\\
\geq&\,\Big|~\sum_{\alpha=1}^{3}\left(C_{1}^{\alpha}-C_{2}^{\alpha}\right)\nabla\bar{u}_{1}^{\alpha}(x)
~\Big|-C.
\end{align}
For $|x_{d}|<\varepsilon/2$,
\begin{equation}\label{C1-C2_ubar}
\sum_{\alpha=1}^{2}\left(C_{1}^{\alpha}-C_{2}^{\alpha}\right)\nabla\bar{u}_{1}^{\alpha}(0',x_{d})
=\frac{1}{\varepsilon}
\begin{pmatrix}
0 &C_{1}^{1}-C_{2}^{1}~\\\\
~0 & C_{1}^{2}-C_{2}^{2}~\\\\
~0 & C_{1}^{3}-C_{2}^{3}~
\end{pmatrix}.
\end{equation}
Therefore, it suffices to obtain a positive lower bound of $\left|C_{1}^{1}-C_{2}^{1}\right|$, $\left|C_{1}^{2}-C_{2}^{2}\right|$ or $\left|C_{1}^{3}-C_{2}^{3}\right|$.

If $b_{*1}^{k_0}\neq0$, for some integer $1\leq\,k_0\leq\,3$, then it follows from Proposition \ref{lemma5.1d=2} that there exists a universal constant $C_{0}>0$ and a sufficiently small number $\varepsilon_0>0$, such that, for $0<\varepsilon<\varepsilon_0$,
\begin{align*}\label{bk0}
|b_{1}^{k_0}|>\frac{1}{C_{0}}.
\end{align*}
By (\ref{C1-C2_d=3}), for sufficiently small $\varepsilon$, 
 \begin{align*}
 |C_{1}^{k_0}-C_{2}^{k_0}|\geq\frac{|b_{1}^{k_0}|}{C|\log\varepsilon|}\geq\frac{|b_{*1}^{k_0}|}{C|\log\varepsilon|}.
 \end{align*}
Combining it with \eqref{nablaud=3} and \eqref{C1-C2_ubar} immediately yields that 
\begin{align*}
|\nabla u(0', x_{d})|\geq\frac{|b_{*1}^{k_0}|}{C\varepsilon|\log\varepsilon|},\quad |x_{d}|<\varepsilon/2.
\end{align*}
Theorem \ref{thm1.2} is thus established.
\end{proof}

\bigskip

\section{Proof of Theorem \ref{thm1.3}}\label{sec6}

\subsection{Decomposition of $u$}

We make use of the following decomposition as in \cite{bly1}, 
\begin{equation}\label{decom_u2}
u=C_{1}v_{1}+C_{2}v_{2}+v_{0},\qquad\mbox{in}~\widetilde{\Omega},
\end{equation}
where $v_{i}\in{C}^{1}(\overline{\widetilde{\Omega}})$, $i=1,2,0$, are, respectively, the solutions of
\begin{equation}\label{v12}
\begin{cases}
\Delta v_{i}=0,&\mbox{in}~\widetilde{\Omega},\\
v_{i}=1,&\mbox{on}~\partial{D}_{i},\\
v_{i}=0,&\mbox{on}~\partial\widetilde{\Omega}\setminus\partial{D}_{i},
\end{cases}
\end{equation}
and
\begin{equation}\label{v0}
\begin{cases}
\Delta v_{0}=0,&\mbox{in}~\widetilde{\Omega},\\
v_{0}=0,&\mbox{on}~\partial{D}_{1}\cup\partial{D}_{2},\\
v_{0}=\varphi,&\mbox{on}~\partial{\Omega}.
\end{cases}
\end{equation}
The constants $C_{i}:=C_{i}(\varepsilon)$, $i=1,2$, in \eqref{decom_u2}, are uniquely determined by $u$.

Denote
$$u^{b}:=C_{2}(v_{1}+v_{2})+v_{0},$$
then $u^{b}$ verifies
\begin{equation}\label{ub}
\begin{cases}
\Delta u^{b}=0,&\mbox{in}~\widetilde{\Omega},\\
u^{b}=C_{2},&\mbox{on}~\partial{D}_{1}\cup\partial{D}_{2},\\
u^{b}=\varphi,&\mbox{on}~\partial{\Omega}.
\end{cases}
\end{equation}
By using theorem 1.1 in \cite{llby} again, we have 
\begin{equation}\label{ub_bdd}
|\nabla u^{b}|\leq\,C,
\end{equation}
where $C$ is a universal constant, independent of $\varepsilon$.  

In view of the decomposition \eqref{decom_u2}, we write
\begin{equation}\label{nablau_dec2}
\nabla{u}=\left(C_{1}-C_{2}\right)\nabla{v}_{1}
+\nabla{u}^{b},\quad\mbox{in}~\widetilde{\Omega}.
\end{equation}
It follows from the forth line of \eqref{mainequation2} that
\begin{align}\label{C1C2_22}
(C_{1}-C_{2})\int_{\partial{D}_{j}}\frac{\partial{v}_{1}}{\partial\nu}\Big|_{+}
+\int_{\partial{D}_{j}}\frac{\partial{u}^{b}}{\partial\nu}\Big|_{+}=0,\quad\,j=1,2.
\end{align}
Denote 
\begin{align*}
a_{ij}:=-\int_{\partial{D}_{j}}\frac{\partial{v}_{i}}{\partial\nu}\Big|_{+},\qquad
b_{j}:=b_{j}[\varphi]=\int_{\partial{D}_{j}}\frac{\partial{u}^{b}}{\partial\nu}\Big|_{+},\qquad\,i,j=1,2.
\end{align*}
Then \eqref{C1C2_22} can be written as
\begin{equation}\label{equ_abc2}
\begin{cases}
 a_{11}(C_{1}-C_{2}) =b_{1},\\
 a_{12}(C_{1}-C_{2}) =b_{2}.
 \end{cases}
\end{equation}

Recalling the definitions of $v_{1},v_{2}$ and by using the integration by parts, we have
$$a_{ij}=-\int_{\partial{D}_{j}}\frac{\partial{v}_{i}}{\partial\nu}\Big|_{+}=\int_{\widetilde\Omega}\nabla v_{i}\cdot\nabla v_{j},~\mbox{and}\quad~b_{j}=\int_{\partial{D}_{j}}\frac{\partial{u}^{b}}{\partial\nu}\Big|_{+}=\int_{\widetilde\Omega}\nabla u^{b}\cdot\nabla v_{j},\,i,j=1,2.$$
By the same argument for the estimate of $|\nabla v_{1}^{1}|$ in Proposition \ref{prop_gradient}, or the estimate of $|\nabla v_{1}|$ in Proposition 3.1 in \cite{lx}, one can see
\begin{equation}\label{estimate_w}
\|\nabla(v_{1}-\bar{u})\|_{L^{\infty}(\Omega_{R/2})}\leq\,C,\quad\|\nabla(v_{2}-\underline{u})\|_{L^{\infty}(\Omega_{R})}\leq\,C.
\end{equation}
So that
\begin{equation}\label{nablav12}
\frac{1}{C(\varepsilon+|x'|^{2})}\leq|\nabla v_{i}(x)|\leq\frac{C}{\varepsilon+|x'|^{2}},\quad\,i=1,2,\quad x\in\Omega_{R},
\end{equation}
and
\begin{equation}\label{v_out}
\|\nabla{v}_{i}\|_{L^{\infty}(\widetilde{\Omega}\setminus\Omega_{R/2})}\leq\,C,\quad\,i=1,2.
\end{equation}
Therefore,
\begin{equation}\label{a11}
\frac{1}{C\rho_{d}(\varepsilon)}\leq a_{11}\leq\frac{C}{\rho_{d}(\varepsilon)},\quad\frac{1}{C\rho_{d}(\varepsilon)}\leq -a_{12}\leq\frac{C}{\rho_{d}(\varepsilon)},
\end{equation}
and in view of \eqref{ub_bdd},
\begin{equation}\label{bj}
|b_{j}|\leq C,\quad\,j=1,2.
\end{equation}

\subsection{Proof of Theorem \ref{thm1.3}}

\begin{proof}[Proof of Theorem \ref{thm1.3}]
If $b_{1}=0$, then by solving \eqref{equ_abc2}, we have $C_{1}-C_{2}=0$, since $a_{11}>0$. Hence, it follows from \eqref{nablau_dec2} and \eqref{ub_bdd} that $|\nabla u|=|\nabla u^{b}|\leq\,C$. Therefore, blow-up does not occur. 

If $b_{1}\neq0$, then solving \eqref{equ_abc2}, we have
\begin{equation}\label{C1-C2b0}
|C_{1}-C_{2}|=\frac{|b_{1}|}{a_{11}}.
\end{equation}
In view of \eqref{a11} and \eqref{bj}, we have
\begin{equation}\label{C1-C2b}
|C_{1}-C_{2}|\leq\,C\rho_{d}(\varepsilon),
\end{equation}
and
\begin{equation}\label{C1-C2blower}
|C_{1}-C_{2}|\geq\,C|b_{1}|\rho_{d}(\varepsilon), \quad\mbox{if}~b_{1}\neq0.
\end{equation}

It follows from \eqref{nablau_dec2}, \eqref{ub_bdd}, \eqref{nablav12} and \eqref{C1-C2b} that
\begin{equation}\label{nablau}
\left|\nabla{u}\right|
\leq\left|C_{1}-C_{2}\right|\left|\nabla{v}_{1}\right|
+\big|\nabla{u}^{b}\big|\leq\frac{C\rho_{d}(\varepsilon)}{\varepsilon+|x'|^{2}},\quad\mbox{in}~\Omega_{R},
\end{equation}
and on the other hand, by \eqref{C1-C2blower},
\begin{equation}\label{nablau_lower}
\left|\nabla{u}(0',x_{d})\right|
\geq\left|(C_{1}-C_{2})\nabla{v}_{1}(0',x_{d})\right|
-C\geq\frac{|b_{1}|\rho_{d}(\varepsilon)}{C\varepsilon},\quad|x_{d}|<\frac{\varepsilon}{2}.
\end{equation}
By using Proposition \ref{lem5.1} below, $b_{j}\rightarrow b_{j}^{*}$, as $\varepsilon\rightarrow0$, it follows from \eqref{nablau_lower} that for sufficiently small $\varepsilon$,
$$|\nabla u(x)|\Big|_{\overline{P_{1}P_{2}}}\geq\frac{|b_{1}|\rho_{d}(\varepsilon)}{C\varepsilon}\geq\frac{|b_{1}^{*}|\rho_{d}(\varepsilon)}{C\varepsilon}.$$
Thus, $b_{1}^{*}$ is a blow-up factor. If $|b_{1}^{*}|\neq0$, then we obtain the lower bound of $|\nabla u(x)|$ on $\overline{P_{1}P_{2}}$ and complete the proof of Theorem \ref{thm1.3}.
\end{proof}

\subsection{The blow-up factors $b_{j}^{*}[\varphi]$}

In order to characterize the limit properties of $b_{1}$ (or $b_{2}$), we consider the following limit problem. Let $u^{*}$ be the solution of
\begin{equation}\label{equ_u*}
\begin{cases}
\Delta u^{*}=0,&\mbox{in}~\widetilde{\Omega}^{*},\\
u^{*}=C^{*},&\mbox{on}~\partial{D}^{*}_{1}\cup\partial{D}^{*}_{2},\\
\int_{\partial{D}_{1}}\frac{\partial{u}^{*}}{\partial\nu}\Big|_{+}+\int_{\partial{D}_{2}}\frac{\partial{u}^{*}}{\partial\nu}\Big|_{+}
=0,&\\
u^{*}=\varphi,&\mbox{on}~\partial{\Omega}.
\end{cases}
\end{equation}
where, by the uniqueness of the solution and \eqref{C1-C2b}, we have $C^{*}=\frac{1}{2}(C_{1}+C_{2})$. Define 
$$b_{j}^{*}:=b_{j}^{*}[\varphi]=\int_{\partial{D}_{j}^{*}}\frac{\partial u^{*}}{\partial\nu}.$$
Then 
\begin{prop}\label{lem5.1}
For $d=2,3$,
\begin{equation}\label{difference_alpha}
|b_{j}-b_{j}^{*}|=\left|\int_{\partial{D}_{j}}\frac{\partial{u}^{b}}{\partial\nu}ds-\int_{\partial{D}_{j}^{*}}\frac{\partial{u}^{*}}{\partial\nu}ds\right|\leq\,C\sqrt{\rho_{d}(\varepsilon)},\quad j=1,2.
\end{equation}
\end{prop}
To prove Proposition \ref{lem5.1}, we need the following lemma.
\begin{lemma}\label{lem6.2}
 Let $C_{1}$ and $C_{2}$ be defined in \eqref{decom_u2} and $C^{*}$ be in \eqref{equ_u*}. We have
\begin{equation}\label{C*}
\left|\frac{C_{1}+C_{2}}{2}-C^{*}\right|\leq\,C\rho_{d}(\varepsilon).
\end{equation}
As a consequence, combining it with \eqref{C1-C2b}, we have
\begin{equation}\label{C*-C2b}
\left|C_{i}-C^{*}\right|\leq\left|C_{i}-\frac{C_{1}+C_{2}}{2}\right|+\left|\frac{C_{1}+C_{2}}{2}-C^{*}\right|\leq\,C\rho_{d}(\varepsilon),\quad\,i=1,2.
\end{equation}
\end{lemma}
The proof of Lemma \ref{lem6.2} will be given later. We first use it to prove Proposition \ref{lem5.1}.

\begin{proof}[Proof of Proposition \ref{lem5.1}]
Recall $V=\Omega\setminus\overline{D_{1}\cup{D}_{2}\cup{D}_{1}^{*}\cup{D}_{2}^{*}}$, and $\Gamma_{i1}:=\partial{D}_{i}^{*}\setminus\partial{D}_{i}$ and $\Gamma_{i2}:=\partial{D}_{i}\setminus\partial{D}_{i}^{*}$, $i=1,2$. Then 
$\partial{V}=\cup_{i,j=1,2}\Gamma_{ij}\cup\partial\Omega$. Setting
$$\phi(x):={u}^{b}(x)-u^{*}(x),$$
then $\Delta\phi=0$ in $V$. It is easy to see that $\phi=0$ on $\partial\Omega$. On $\Gamma_{11}$, by mean value theorem, \eqref{ub_bdd} and \eqref{C*-C2b}, we have
$$|\phi|\Big|_{\Gamma_{11}}=|u^{b}-u^{*}|\Big|_{\Gamma_{11}}=|C_{2}+|\nabla u^{b}(\xi)|\varepsilon-C^{*}|\leq\,C\rho_{d}(\varepsilon)+C\varepsilon\leq\,C\rho_{d}(\varepsilon),$$
where $\xi\in{D}_{1}^{*}\setminus{D}_{1}$. Similarly,
$$|\phi|\Big|_{\Gamma_{12}}=|u^{b}-u^{*}|\Big|_{\Gamma_{12}}\leq|C_{2}-C^{*}-|\nabla u^{*}(\xi)|\varepsilon|\leq\,C\rho_{d}(\varepsilon)+C\varepsilon\leq\,C\rho_{d}(\varepsilon),$$
for some $\xi\in{D}_{1}\setminus{D}_{1}^{*}$. By the same way,
$$|\phi|\Big|_{\Gamma_{21}\cup\Gamma_{22}}\leq\,C\rho_{d}(\varepsilon).$$
We now apply the maximum principle to $\phi$ on $V$, 
\begin{equation}\label{phi_bdd}
|\phi|\leq\,C\rho_{d}(\varepsilon),\quad\mbox{on}~~V.
\end{equation}

Denote
$$\Omega^{+}:=V\cap\{x\in\Omega~|~x_{d}>0\},\quad\mbox{and}~~\partial\Omega^{+}:=\{x\in\partial\Omega~|~x_{d}>0\},$$
and $\gamma=\{x_{d}=0\}\cap\Omega$. Since ${u}^{b}$ and $u^{*}$ are harmonic in $\Omega^{+}\setminus{D}_{1}$ and $\Omega^{+}\setminus{D}_{1}^{*}$, respectively, by using integration by parts, we have
$$\int_{\partial{D}_{1}}\frac{\partial{u}^{b}}{\partial\nu}=\int_{\partial\Omega^{+}}\frac{\partial{u}^{b}}{\partial\nu}+\int_{\gamma}\frac{\partial{u}^{b}}{\partial\nu},$$
and
$$\int_{\partial{D}_{1}^{*}}\frac{\partial{u}^{*}}{\partial\nu}=\int_{\partial\Omega^{+}}\frac{\partial{u}^{*}}{\partial\nu}+\int_{\gamma}\frac{\partial{u}^{*}}{\partial\nu}.$$
Thus,
$$\int_{\partial{D}_{1}}\frac{\partial{u}^{b}}{\partial\nu}-\int_{\partial{D}_{1}^{*}}\frac{\partial{u}^{*}}{\partial\nu}ds=\int_{\partial\Omega^{+}}\frac{\partial\phi}{\partial\nu}+\int_{\gamma}\frac{\partial\phi}{\partial\nu}.$$

Divide $\gamma$ into three pieces: $\gamma=\gamma_{1}\cup\gamma_{2}\cup\gamma_{3}$, where
$$\gamma_{1}:=\{(x',0)~|~|x'|\leq\,\sqrt{\rho_{d}(\varepsilon)}\},\quad\gamma_{2}:=\{(x',0)~|~\sqrt{\rho_{d}(\varepsilon)}<|x'|<R\},$$
$$\gamma_{3}:=\gamma\setminus(\gamma_{1}\cup\gamma_{2}).$$
Write
$$\int_{\gamma}\frac{\partial\phi}{\partial\nu}=\int_{\gamma_{1}}+\int_{\gamma_{2}}+\int_{\gamma_{3}}\frac{\partial\phi}{\partial\nu}:=\mathrm{I}+\mathrm{II}+\mathrm{III}.$$
First, for $(y',0)\in\gamma_{1}$, since $|\nabla u^{b}|,|\nabla u^{*}|\leq\,C$ in $\Omega_{R}$, so $|\nabla\phi|\leq\,C$ in $\Omega_{R}$. Hence
$$|\mathrm{I}|\leq\,C(\sqrt{\rho_{d}(\varepsilon)})^{d-1}\leq\,C\rho_{d}(\varepsilon)^{(d-1)/2}.$$
For $(y',0)\in\gamma_{2}$, there exists a $r>\frac{1}{C}|y'|^{d}$ for some $C>1$ such that $B_{r}(y',0)\subset{V}$. It then follows from the standard gradient estimates for harmonic function and \eqref{phi_bdd} that
$$|\nabla\phi(y',0)|\leq\frac{C\rho_{d}(\varepsilon)}{|y'|^{d}},$$
and
$$|\mathrm{II}|\leq\,C\rho_{d}(\varepsilon)\int_{\sqrt{\rho_{d}(\varepsilon)}<|y'|<R}\frac{1}{|y'|^{d}}dS\leq\,C\sqrt{\rho_{d}(\varepsilon)}.$$
For $(y',0)\in\gamma_{3}$, there is a universal constant $r>0$ such that $B_{r}(x)\subset{V}$ for all $x\in\gamma_{3}$. So we have from \eqref{phi_bdd} that for any $x\in\gamma_{3}$,
$$|\nabla\varphi|\leq\frac{C\varepsilon}{r}\leq\,C\varepsilon,$$
and
$$|\mathrm{III}|\leq\,C\varepsilon.$$
Finally, using the standard boundary gradient estimates for $\phi$ and \eqref{phi_bdd}, we have
$$\Big|\int_{\partial\Omega^{+}}\frac{\partial\phi}{\partial\nu}\Big|\leq\,C\rho_{d}(\varepsilon).$$
Thus, we have \eqref{difference_alpha}. The proof is completed.
\end{proof}

\begin{remark}
From the proof of Proposition \ref{lem5.1}, one can see that the convergence of $b_{j}$ depends on the estimates \eqref{C1-C2b}, $|C_{1}-C_{2}|\leq\,C\rho_{d}(\varepsilon)\rightarrow0$. However, in higher dimensions $d\geq4$, we still do not have such closeness of $C_{1}$ and $C_{2}$. However, from \eqref{C1-C2b0}, it is not difficult to find a boundary data $\varphi$ such that $|b_{1}[\varphi]|\geq\frac{1}{C}$ for some universal constant $C$, although $b_{1}^{*}[\varphi]$ is not necessarily its limit. Thus, we also can have a lower bound estimate, $|\nabla u(x)|_{\overline{P_{1}P_{2}}}\geq\frac{1}{C\varepsilon}$. 
\end{remark}

\begin{proof}[Proof of Lemma \ref{lem6.2}]
We decompose $u^{*}$ into
$$u^{*}=C^{*}v_{1}^{*}+v_{0}^{*},\quad\mbox{in}~\widetilde{\Omega}^{*},$$
where $v_{1}^{*},v_{0}^{*}\in{C}^{1}(\overline{\widetilde{\Omega}})$ are, respectively, the solutions of
\begin{equation}\label{v1*}
\begin{cases}
\Delta v_{1}^{*}=0,&\mbox{in}~\widetilde{\Omega}^{*},\\
v_{1}^{*}=1,&\mbox{on}~\partial{D}_{1}^{*}\cup\partial{D}_{2}^{*},\\
v_{1}^{*}=0,&\mbox{on}~\partial\Omega,
\end{cases}
\end{equation}
and
\begin{equation}\label{v0*}
\begin{cases}
\Delta v_{0}^{*}=0,&\mbox{in}~\widetilde{\Omega}^{*},\\
v_{0}^{*}=0,&\mbox{on}~\partial{D}_{1}^{*}\cup\partial{D}_{2}^{*},\\
v_{0}^{*}=\varphi,&\mbox{on}~\partial{\Omega}.
\end{cases}
\end{equation}
From the third line of \eqref{equ_u*}, we have
\begin{equation}\label{equ_C*}
C^{*}\left(\int_{\partial{D}_{1}^{*}}\frac{\partial{v}_{1}^{*}}{\partial\nu}\Big|_{+}+\int_{\partial{D}_{2}^{*}}\frac{\partial{v}_{1}^{*}}{\partial\nu}\Big|_{+}\right)+\left(\int_{\partial{D}_{1}^{*}}\frac{\partial{v}_{0}^{*}}{\partial\nu}\Big|_{+}+\int_{\partial{D}_{2}^{*}}\frac{\partial{v}_{0}^{*}}{\partial\nu}\Big|_{+}\right)
=0
\end{equation}
Let $v_{1},v_{2}$ and $v_{0}$ be defined in \eqref{v12} and \eqref{v0}. We claim that
\begin{equation}\label{convergence_v12}
\left|\int_{\partial{D}_{i}}\frac{\partial({v}_{1}+v_{2})}{\partial\nu}\Big|_{+}-\int_{\partial{D}_{i}^{*}}\frac{\partial{v}_{1}^{*}}{\partial\nu}\Big|_{+}\right|\leq\,C\rho_{d}(\varepsilon),\quad\,i=1,2,
\end{equation}
and
\begin{equation}\label{convergence_v0}
\left|\int_{\partial{D}_{i}}\frac{\partial{v}_{0}}{\partial\nu}\Big|_{+}-\int_{\partial{D}_{i}^{*}}\frac{\partial{v}_{0}^{*}}{\partial\nu}\Big|_{+}\right|\leq\,C\rho_{d}(\varepsilon),\quad\,i=1,2.
\end{equation}

As in the proof of Proposition \ref{lem5.1}, letting
$$\phi_{1}:=(v_{1}+v_{2})-v^{*1},$$
then $\Delta\phi_{1}=0$ in $V$, and $\phi_{1}=0$ on $\partial\Omega$.
Since $v_{1}+v_{2}$ satisfies
\begin{equation}\label{v1+v2}
\begin{cases}
\Delta (v_{1}+v_{2})=0,&\mbox{in}~\widetilde{\Omega},\\
v_{1}+v_{2}=1,&\mbox{on}~\partial{D}_{1}\cup\partial{D}_{2},\\
v_{1}+v_{2}=0,&\mbox{on}~\partial{\Omega},
\end{cases}
\end{equation}
it follows from theorem 1.1 in \cite{llby} that
\begin{equation}\label{nabla(v_{1}+v_{2})}
|\nabla(v_{1}+v_{2})|\leq\,C,\quad\mbox{in}~~\widetilde{\Omega},
\end{equation}
since the potential takes the same constant value on boundaries of both partials. Because of the same reason,
\begin{equation}\label{nabla{v}^{*1}}
|\nabla{v}_{1}^{*}|\leq\,C,\quad\mbox{in}~~\widetilde{\Omega}^{*}.
\end{equation}
Thus, on $\Gamma_{11}$, by using mean value theorem and \eqref{nabla(v_{1}+v_{2})}, we have
\begin{align*}
|\phi_{1}|\Big|_{\Gamma_{11}}&=|(v_{1}+v_{2})-v_{1}^{*}|\Big|_{\Gamma_{11}}
=|(v_{1}+v_{2})-1|\Big|_{\Gamma_{11}}\\
&=|(v_{1}+v_{2})-(v_{1}+v_{2})(x',x_{d}+\varepsilon)|\Big|_{\Gamma_{11}}\\
&\leq|\nabla (v_{1}+v_{2})(\xi)|\varepsilon\leq\,C\varepsilon,
\end{align*}
for some $\xi\in\widetilde{\Omega}$; similarly, using \eqref{nabla{v}^{*1}},
\begin{align*}
|\phi_{1}|\Big|_{\Gamma_{12}}&=|(v_{1}+v_{2})-v_{1}^{*}|\Big|_{\Gamma_{12}}=|1-v_{1}^{*}|\Big|_{\Gamma_{12}}\\
&=|v_{1}^{*}(x',x_{d}-\varepsilon)-v_{1}^{*}|\Big|_{\Gamma_{12}}=|\nabla v_{1}^{*}(\xi)|\varepsilon\leq\,C\varepsilon,
\end{align*}
for some another $\xi\in\widetilde{\Omega}^{*}$. By the same way,
$$|\phi_{1}|\Big|_{\Gamma_{21}\cup\Gamma_{22}}\leq\,C\varepsilon.
$$
We now apply the maximum principle to $\phi_{1}$ on $V$, instead of \eqref{phi_bdd}, we have
\begin{equation}\label{phi1_bdd}
|\phi_{1}|\leq\,C\varepsilon,\quad\mbox{on}~~V.
\end{equation}
Therefore, by the same process as in the rest of the proof of Proposition \ref{lem5.1}, \eqref{convergence_v12} for $i=1$ is proved. The proofs of claim \eqref{convergence_v12} for $i=2$ and \eqref{convergence_v0} are similar.

In view of the decomposition \eqref{decom_u2}, the forth line of \eqref{mainequation2}, we have
\begin{align}\label{system_C1C2}
C_{1}\int_{\partial{D}_{j}}\frac{\partial{v}_{1}}{\partial\nu}\Big|_{+}
+C_{2}\int_{\partial{D}_{j}}\frac{\partial{v}_{2}}{\partial\nu}\Big|_{+}+\int_{\partial{D}_{j}}\frac{\partial{v}^{0}}{\partial\nu}\Big|_{+}=0,\quad\,j=1,2.
\end{align}
That is,
$$\begin{cases}
a_{11}C_{1}+a_{12}C_{2}=\tilde{b}_{1},\\
a_{21}C_{1}+a_{22}C_{2}=\tilde{b}_{2}.
\end{cases}$$
So that
$$(a_{11}+a_{21})C_{1}+(a_{12}+a_{22})C_{2}+\tilde{b}_{1}+\tilde{b}_{2}=0.$$
Since $a_{12}=a_{21}$, it follows that
$$(a_{11}+a_{21})(C_{1}+C_{2})+(a_{22}-a_{11})C_{2}+\tilde{b}_{1}+\tilde{b}_{2}=0.$$
Similarly,
$$(a_{12}+a_{22})(C_{1}+C_{2})-(a_{22}-a_{11})C_{1}+\tilde{b}_{1}+\tilde{b}_{2}=0.$$
Adding these two equations together and dividing it by two yields
\begin{equation*}
(a_{11}+a_{21}+a_{12}+a_{22})\frac{(C_{1}+C_{2})}{2}+(a_{22}-a_{11})\frac{(C_{2}-C_{1})}{2}+\tilde{b}_{1}+\tilde{b}_{2}=0.
\end{equation*}
That is,
\begin{align}\label{C1+C2/2}
&\left(\int_{\partial{D}_{1}}\frac{\partial({v}_{1}+v_{2})}{\partial\nu}\Big|_{+}+\int_{\partial{D}_{2}}\frac{\partial({v}_{1}+v_{2})}{\partial\nu}\Big|_{+}\right)\frac{(C_{1}+C_{2})}{2}\nonumber\\
&+\left(\int_{\partial{D}_{1}}\frac{\partial{v}_{0}}{\partial\nu}\Big|_{+}+\int_{\partial{D}_{2}}\frac{\partial{v}_{0}}{\partial\nu}\Big|_{+}\right)+\left(\int_{\widetilde{\Omega}}|\nabla{v}_{2}|^{2}-\int_{\widetilde{\Omega}}|\nabla{v}_{1}|^{2}\right)\frac{(C_{2}-C_{1})}{2}=0.
\end{align}
Recalling $\underline{u}=1-\bar{u}$ in $\Omega_{R}$ and using  estimates \eqref{estimate_w} and \eqref{v_out} leads to
\begin{align*}
&\left|\int_{\widetilde{\Omega}}|\nabla{v}_{2}|^{2}-\int_{\widetilde{\Omega}}|\nabla{v}_{1}|^{2}\right|\\
&\leq\left|\int_{\Omega_{R}}|\nabla{v}_{2}|^{2}-\int_{\Omega_{R}}|\nabla{v}_{1}|^{2}\right|+\left|\int_{\widetilde{\Omega}\setminus\Omega_{R}}|\nabla{v}_{2}|^{2}-\int_{\widetilde{\Omega}\setminus\Omega_{R}}|\nabla{v}_{1}|^{2}\right|\\
&\leq\left|\int_{\Omega_{R}}|\nabla\underline{u}|^{2}-\int_{\Omega_{R}}|\nabla\bar{u}|^{2}\right|+\left|\int_{\Omega_{R}}|\nabla({v}_{2}-\underline{u})|^{2}-\int_{\Omega_{R}}|\nabla({v}_{1}-\bar{u})|^{2}\right|+C\\
&\leq\,C.
\end{align*}
By using the integration by parts and the definition of $v_{1}^{*}$, we have
$$\int_{\partial{D}_{1}^{*}}\frac{\partial{v}_{1}^{*}}{\partial\nu}\Big|_{+}+\int_{\partial{D}_{2}^{*}}\frac{\partial{v}_{1}^{*}}{\partial\nu}\Big|_{+}=\int_{\widetilde{\Omega}^{*}}|\nabla v_{1}^{*}|^{2}>0.$$Therefore, by the claim above, \eqref{C1+C2/2} can be written as
\begin{align*}
&\left(\int_{\partial{D}_{1}^{*}}\frac{\partial{v}_{1}^{*}}{\partial\nu}\Big|_{+}+\int_{\partial{D}_{2}^{*}}\frac{\partial{v}_{1}^{*}}{\partial\nu}\Big|_{+}+O(\rho_{d}(\varepsilon))\right)\frac{(C_{1}+C_{2})}{2}\nonumber\\
&+\left(\int_{\partial{D}_{1}^{*}}\frac{\partial{v}_{0}^{*}}{\partial\nu}\Big|_{+}+\int_{\partial{D}_{2}^{*}}\frac{\partial{v}_{0}^{*}}{\partial\nu}\Big|_{+}+O(\rho_{d}(\varepsilon))\right)+O(\rho_{d}(\varepsilon))=0.
\end{align*}
Comparing it with \eqref{equ_C*}, the proof of \eqref{C*} is finished.
\end{proof}

\noindent{\bf{\large Acknowledgements.}}
The author is greatly indebted to Professor YanYan Li and JiGuang Bao for their constant encouragement and support. The author was partially supported by  NSFC (11571042, 11631002),  Fok Ying Tung Education Foundation (151003).

\vspace{5mm}


\bibliographystyle{amsplain}
\bibliography{References}

\end{document}